\documentclass[12pt]{article}
\usepackage{amsfonts}
\usepackage{amsthm}
\usepackage{latexsym}
\usepackage{amssymb}
\usepackage{amstext}
\usepackage{amsmath}
\usepackage{rotating}
\usepackage{graphicx}
\usepackage{amssymb}
\usepackage{epstopdf}
\usepackage{epsfig}
\usepackage{rotating}
\usepackage{appendix}

\newtheorem{thm}{Theorem}
\newtheorem{prop}{Proposition}

\newtheorem{cor}{Corollary}
\newtheorem{lem}{Lemma}

\newtheorem{example}{Example}

\def\Z{{\mathbb Z}}

\def\G{{\mathcal G}}
\def\H{{\mathcal H}}

\def\G{\overline{G}}
\def\hat G{\widehat{G}}
\def\H{\overline{H}}

\def\P{{\overline P}}
\def\hatP{{\widehat P}}
\def\C{{\overline C}}

\def\calP{{\mathcal P}}
\def\calC{{\mathcal C}}
\def\calK{{\mathcal K}}

\def\1{\bigskip\noindent}
\def\2{\bigskip}

\def\example{\medskip\noindent{\bf Example.} }
\def\definition{\medskip\noindent{\bf Definition.} }

\begin{document}

\begin{center}
\LARGE {\bf \textsc{Posets of Geometric Graphs}} 
\end{center}\bigskip

\begin{center}
\textsc{Debra L. Boutin} 

\textsc{Department of Mathematics}

\textsc{Hamilton College,  Clinton,  NY 13323}

\textsc {\sl dboutin@hamilton.edu}

\end{center}

\begin{center}
\textsc{Sally Cockburn}

\textsc{Department of Mathematics}

\textsc{Hamilton College,  Clinton,  NY 13323}

\textsc {\sl scockbur@hamilton.edu}

\end{center}

\begin{center}
\textsc{Alice M. Dean}

\textsc{Mathematics and Computer Science Department}

\textsc{Skidmore College,  Saratoga Springs,  NY 12866}

\textsc {\sl adean@skidmore.edu}

\end{center}

\begin{center}
\textsc{Andrei Margea}

\textsc{Department of Computer Science}

\textsc{University of Texas,  Austin,  TX 78712}

 \textsc {\sl utistu87@cs.utexas.edu}

\end{center}

\begin{abstract} A {\it geometric graph} $\G$ is a simple graph drawn in the plane,  on points in general position, with straight-line edges.   We call $\G$  a {\it geometric realization} of the underlying abstract graph $G$.  A {\it geometric homomorphism} $f:\overline{G} \to \overline{H}$ is a vertex map that preserves adjacencies and crossings (but not necessarily non-adjacencies or non-crossings).  This work uses geometric homomorphisms to introduce a partial order on the set of isomorphism classes of geometric realizations of an abstract graph $G$.  Set $\G\preceq \hat G$ if $\G$ and $\hat G$ are geometric realizations of $G$ and there is a vertex-injective geometric homomorphism $f:\G \to \hat G$.   This paper develops tools to determine when two geometric realizations are comparable.  Further, for $3\leq n \leq 6$, this paper provides the isomorphism classes of geometric realizations of $P_n,  C_n$ and $K_n$, as well as the Hasse diagrams of the  geometric homomorphism posets (resp., $\calP_n, \calC_n, \calK_n$)  of these graphs.   The paper also provides the following results for general $n$:  each of $\calP_n$ and $\calC_n$ has a unique minimal element and a unique maximal element; if $k\leq n$ then $\calP_k$  (resp., $\calC_k$) is a subposet of $\calP_n$ (resp., $\calC_n$); and $\calK_n$ contains a chain of length $n-2$.      \end{abstract}


\section{Introduction}

The topic of graph homomorphisms has been a subject of growing interest; for an excellent survey of the area see \cite{HN}.  In \cite{BC}, Boutin and Cockburn extend the theory of graph homomorphisms to geometric graphs.  In this paper we use geometric homomorphisms to define and study posets of geometric realizations of a given abstract graph. Throughout this work the term graph means simple graph.

A geometric graph $\overline{G}$ is a graph drawn in the plane, on points in general position, with straight-line edges.  What we care about in a geometric graph is which pairs of vertices are adjacent and which pairs of edges cross.  In particular, two geometric graphs are said to be isomorphic if there is a bijection between their vertex sets that preserves adjacencies, non-adjacencies, crossings, and non-crossings.

A natural way to extend the idea of abstract graph homomorphism to the context of geometric graphs is to define a geometric homomorphism as a vertex map $f:\overline{G} \to \overline{H}$ that preserves both adjacencies and crossings (but not necessarily non-adjacencies or non-crossings).  If such a map exists we write $\overline{G}\to\overline{H}$ and say that $\overline{G}$ is (geometrically) homomorphic to $\overline{H}$.  There are many similarities between abstract graph homomorphisms and geometric graph homomorphisms, but there are also great contrasts.  Results that are straightforward in the context of abstract graphs can become complex in the context of geometric graphs.  

In abstract graph homomorphism theory, two vertices cannot be identified under any homomorphism if and only if they are adjacent.  
However, in \cite{BC} we show that there are  additional reasons why two vertices cannot be identified under any geometric homomorphism:
 if they are involved in a common edge crossing;  if they are endpoints of an odd length path each edge of which is crossed by a common edge; if they are endpoints of a path of length two whose edges cross all edges of an odd length cycle.
 This list is likely not exhaustive.

In abstract graph homomorphism theory, a graph is not homomorphic to a graph on fewer vertices if and only if it is a complete graph.  In geometric homomorphism theory, there are many graphs other than complete graphs that are not homomorphic to any geometric graph on fewer vertices.  For example, since vertices involved in a common crossing cannot be identified by any  geometric homomorphism, there is no geometric homomorphism of a non-plane realization of  $C_4$ into a geometric graph on fewer than four vertices.

In abstract graph homomorphism theory, every graph on $n$ vertices is homomorphic to $K_n$.  However, not all geometric graphs on $n$ vertices are homomorphic to a given realization of $K_n$.  In fact, two different geometric realizations of $K_n$ are not necessarily homomorphic to each other.    For example, consider the three geometric realizations of $K_6$ given in Figure~\ref{fig:three_k6}.
Note that $\G_2$ has a vertex with all incident edges crossed, while $\G_3$  does not;  this can be used to prove that there is no geometric homomorphism from $\G_2$ to $\G_3$.  Also, $\G_3$  has more crossings than $\G_2$; this can be used to prove that there is no geometric homomorphism from $\G_3$ to $\G_2$. On the other hand we can easily argue that while there is no geometric homomorphism from $\G_2$ to $\G_1$, the map $f:\G_1\to \G_2$ implied by the given vertex numbering schemes is a geometric homomorphism.

  \begin{figure}[htbp] 
\centerline{\epsfig{file = 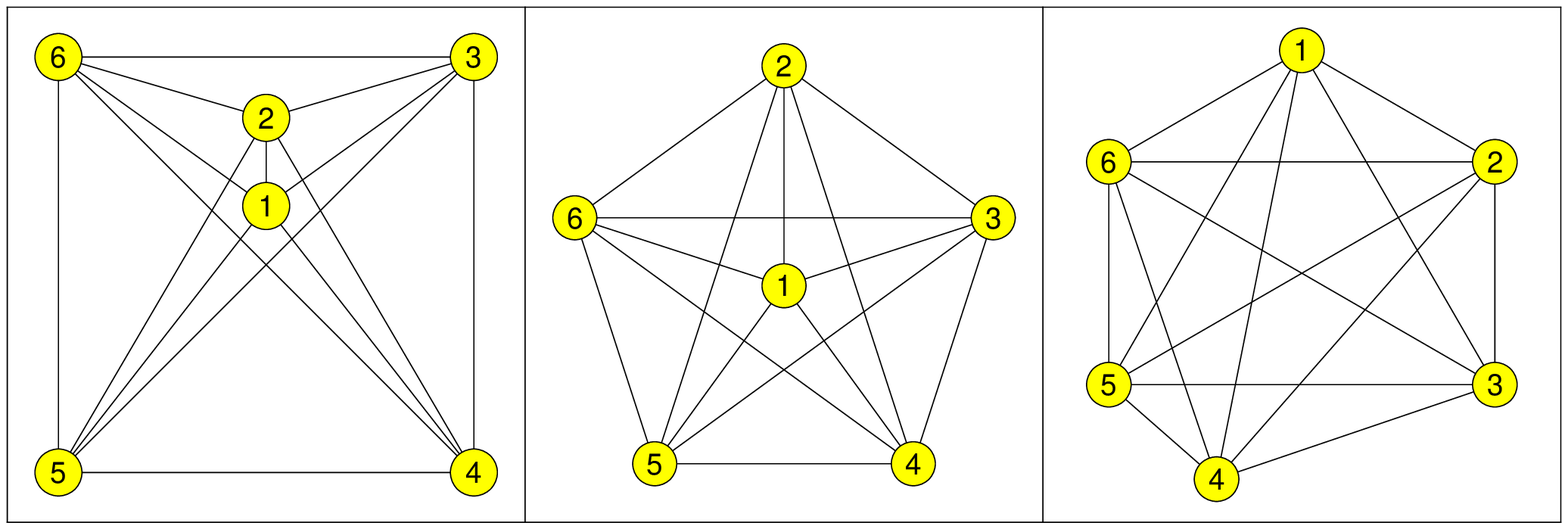,   width = .8\textwidth}} 
  \hspace{1.25in}$\G_1$\hspace{1.25in} $\G_2$\hspace{1.25in} $\G_3$
     \caption{Three geometric realizations of $K_6$}
     \label{fig:three_k6}
  \end{figure}

Since homomorphisms are reflexive and transitive, it is natural to want to use them to induce a partial order.  That is, we would like to define $\G\preceq \H$ when $\G\to \H$.    However, homomorphisms are not necessarily antisymmetric.  It is easy to find geometric (or abstract) graphs $\G$ and $\H$ so that $\G\to \H$ and $\H\to \G$ but $\G$  and  $\H$ are not isomorphic. For example, let $\H$ be any geometric graph with a non-isolated vertex $z$.  Add a vertex $x$ and edge $e$ between $x$ and $z$, positioned so that $e$ crosses no other edge of $\H$.  Call this new graph $\G$.  Identifying $x$ with any neighbor of $z$ gives us $\G \to \H$.  The fact that  $\H$ is a subgraph of $\G$ gives us $\H\to \G$. But clearly $\G$ and  $\H$ are not isomorphic. Thus  graph homomorphisms (whether abstract or geometric) do not induce a partial order since they are not antisymmetric.

In \cite{HN}, Hell and Ne{\v{s}}et{\v{r}}il  solve this problem for abstract graphs by using homomorphisms to define a partial order on the class of non-isomorphic cores of graphs. The {\it core} of an (abstract or geometric) graph is the smallest subgraph to which it is homomorphic. In the example above, $\G$ and $\H$ have isomorphic cores. In this paper  we  solve the problem  by using geometric homomorphisms to define a partial order on the set of geometric realizations of a given abstract graph.  That is to say, we let $\G\preceq \H$ if there is a geometric homomorphism $f:\G\to \H$ that induces an isomorphism on the underlying abstract graphs.  This definition ensures that $\preceq$ is antisymmetric.

The paper is organized as follows. In Section \ref{sect:basics} we give  formal definitions and develop tools that help determine whether two geometric realizations are homomorphic.  In Section \ref{sect:paths} we determine the isomorphism classes and resulting poset, $\calP_n$, of realizations of the path $P_n$ with $2\leq n\leq 6$. Additionally we provide the following  results:  $\calP_n$ has a unique minimal and a unique maximal element; if $k\leq n$, then $\calP_k$ is a subposet of $\calP_n$; for each positive integer $c$ less than or equal to the maximum number of crossings, there is at least one realization of $P_n$ with precisely $c$ crossings.  In Section \ref{sect:cycles} we determine the isomorphism classes and resulting poset, $\calC_n$, of realizations of the cycle $C_n$ with $3\leq n\leq 6$. We also show that  $\calC_n$ has a unique minimal element and a unique minimal element.  Further if $k\leq n$, then $\calC_k$ is a subposet of $\calC_n$.  In Section \ref{sect:cliques} we determine the isomorphism classes and resulting poset, $\calK_n$, of realizations of the complete graph $K_n$ with $3\leq n\leq 6$, and we prove that for all $n$,  $\calK_n$ contains a chain of length $n-2$.  In Section \ref{sect:questions} we provide some open questions.

\section{Basics, Tools, Examples}\label{sect:basics}

A {\it geometric graph} $\overline{G}$ is a simple graph $G = \big( V(G), E(G) \big)$ together with a straight-line drawing of  $G$  in the plane with vertices in general position (no three vertices are collinear and no three edges cross at a single point). A geometric graph $\G$ with underlying abstract graph $G$ is  called a {\it geometric realization} of $G$; the term {\em rectilinear drawing} is also used in the literature. Two geometric realizations of a graph are considered the same if they have the same vertex adjacencies and edge crossings.  This is formalized below by extending the definition of graph isomorphism in a natural way to geometric graphs.

\definition
A geometric isomorphism, denoted $f:\overline{G} \to \overline{H}$, is a  bijection $f:V(\overline G) \to V(\overline H)$ such that for all $u, v, x, y \in V(\overline G)$,
	\begin{enumerate}
	\item $uv \in E(\overline G)$ if and only if $f(u)f(v) \in E(\overline H)$, and
	\item $xy$ crosses $uv$ in $\overline{G}$ if and only if $f(x)f(y)$ crosses $f(u)f(v)$ in $\overline{H}$.
	\end{enumerate} 

If there exists a geometric isomorphism $f:\overline{G} \to \overline{H}$, we write $\G \cong \overline{H}$.  Geometric isomorphism clearly defines an equivalence relation on the set of  geometric realizations of a simple graph $G$.

Note that in \cite{HH, HT} Harborth, et al., give a definition for  isomorphism  of geometric graphs that is stricter than the one given here. They require that a geometric 
isomorphism also preserve regions and parts of edges.  Figure~\ref{fig:two_c6} shows two geometric realizations of $C_6$ that have the same crossings (and so are isomorphic by our definition) but have different regions (and so are not isomorphic in the sense of Harborth).  One consequence of this is that for a given abstract graph there are (potentially) fewer isomorphism classes of realizations under our definition than under that of Harborth. 

\begin{figure}[htbp]
{\epsfig{file = 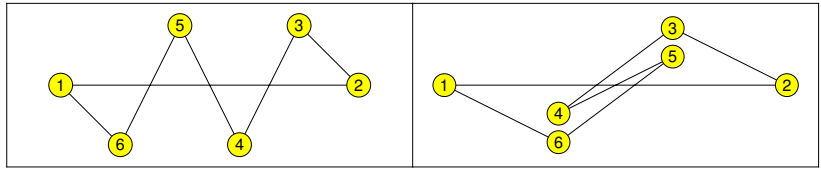,   width = .8\textwidth}} 
  \centering
  \caption{Realizations with the same crossings but different regions}
  \label{fig:two_c6}
\end{figure}

We similarly extend the definition of graph homomorphism to geometric graphs.
	
\definition
 A geometric homomorphism, denoted $f:\overline{G} \to \overline{H}$, is a  function $f:V(\overline  G) \to V( \overline  H)$ such that for all $u, v, x, y \in V(\overline  G)$,
	\begin{enumerate}
	\item if $uv \in E(\overline  G)$, then $ f(u)f(v) \in E(\overline  H)$, and
	\item if $xy$ crosses $uv$ in $\overline{G}$, then  $f(x)f(y)$ crosses $f(u)f(v)$ in $\overline{H}$.
	\end{enumerate} 

If there is a geometric homomorphism $f:\G\to \H$,  we write $\G \stackrel{f}{\to} \H$ or simply $\G\to \H$, and we  say that $\G$ is (geometrically) homomorphic to $\H$.

\definition Let $\overline{G}$ and $ \widehat{G}$ be geometric realizations of a  graph $G$. Set $\overline{G} \preceq \widehat{G}$ (or $\G \stackrel{f}{\preceq}\hat G$) if  there exists a (vertex) injective geometric homomorphism $f:\overline{G} \to \widehat{G}$.

Note that since the abstract graphs underlying $\G$ and $\hat G$ are the same, the fact that $f$ is 
injective and preserves adjacency means that $f$ induces an isomorphism from $G$ to itself.  It is not difficult to see that this relation is reflexive, transitive, antisymmetric, and hence a partial order.

\definition The {\it geometric homomorphism poset} of a  graph  is the set of geometric isomorphism classes of its realizations partially ordered by the relation $\preceq$.

\subsection{The Edge Crossing and Line\slash Crossing Graphs}\label{subsec:LEX}

Recall that the {\it line graph} of an abstract graph $G$, denoted $L(G)$, is the abstract graph whose vertices correspond to edges of $G$, with adjacency when the corresponding edges of $G$ 
are adjacent.
In this section, we define the {\it edge crossing graph}, which is similar to the line graph except that it encodes edge crossings rather than edge adjacencies.  In \cite{N} Ne{\v{s}}et{\v{r}}il, et al.,   proved a correspondence between graph homomorphisms and  homomorphisms of their line graphs.  We generalize this to geometric graphs and the union of their line and crossing graphs.

\definition Let $\G$ be a geometric graph.  Define the {\it edge crossing graph}, denoted $EX(\G)$, to be the abstract graph whose vertices correspond to edges of $\G$, with  adjacency when the corresponding edges of $\G$ cross. Define the {\it line\slash crossing graph}, denoted $LEX(\G)$, to be the 2-edge-colored abstract graph whose vertices correspond to the edges of $\G$, with red edges in $LEX(\G)$ corresponding to adjacent edge pairs
 in $\G$ and blue edges in $LEX(\G)$ corresponding to crossing edge pairs in $\G$.

In other words, the line\slash crossing graph of $\G$ is the union of the line graph of $G$ and the edge crossing graph of $\G$, with an added edge coloring to keep the meanings of these edges clear.  Figure \ref{fig:P6} shows a geometric realization of $P_6$ and its line\slash crossing graph.

The following theorem is well-known and important to our proofs.

\begin{thm}\label{thm:linegraph}\rm \cite{W}  If $G$ has more than four vertices, then $G\cong H \iff L(G)\cong L(H)$.  \end{thm}

\begin{prop} \label{prop:preceqiff}\rm Let $\G$ and $\H$ be geometric graphs on more than four vertices.  Then $\G\preceq\H$  if and only if there exists a color-preserving graph homomorphism $\tilde f: LEX(\G) \to LEX(\H)$ that restricts to an isomorphism from $L(G)$ to $L(H)$.\end{prop}

Note that using  Proposition~\ref{prop:preceqiff} requires that $\G$ and $\H$ have isomorphic underlying abstract graphs. Thus we may assume that  $\G$ and $\H$ are geometric realizations of the same graph. 

An alternate way to phrase Proposition \ref{prop:preceqiff} is to  say that  if $\G$ and $\hat G$ are geometric realizations of the abstract graph $G$, then $\G\preceq \hat G$  if and only if there exists an isomorphism of the line graphs that induces  a homomorphism of the edge crossing graphs. Thus if there is no injective homomorphism $\tilde f:EX(\G) \to EX(\hat G)$, then $\G\not\preceq \hat G$.   Thus given $\G$ and $\hat G$, if $EX(\G)$ is not isomorphic to a subgraph of $EX(\H)$, then there is no geometric homomorphism  $\G \to \hat G$.  However, Proposition \ref{prop:preceqiff} tells us something stronger.  The proposition tells us that $\G\to \hat G$ if and only if there is $\tilde f\in {\rm Aut}(L(G))$ so that $\tilde f (EX(\G))$ is a subgraph of $EX(\hat G)$.   For graphs whose line graphs have small automorphism groups, this reduces the work significantly.



\example  In Figure \ref{fig:P6} we see a realization $\overline P_6$ of $P_6$.  Note that $L(P_6) \cong P_5$ and $EX(\overline P_6) \cong P_2$. 
Each non-plane geometric realization of $P_6$ has edge crossing graph with at least one edge, so theoretically there are many possible realizations $\widehat P_6$ to which $\P_6$ might be geometrically homomorphic.  However, by Proposition \ref{prop:preceqiff} a vertex injective homomorphism $f:\overline P_6 \to \widehat P_6$, taking $EX(\overline P_6)$ to a subgraph of $EX(\widehat P_6)$, needs to be an automorphism of $L(P_6)$. 
Recall that $\P_6$ as given in Figure \ref{fig:P6} has a single crossing that occurs between edges $e_1$ (with endpoints $1$ and $2$) and $e_3$ (with endpoints $3$ and $4$).  Applying the two automorphisms of $L(P_6)$ 
 to this realization, we see that $\P_6 \to \hatP_6$ if and only if $\hatP_6$ has one of the crossings $e_1\times e_3$ or $e_3 \times e_5$.  This restricts our search significantly.

\begin{figure}[htbp]
 \centerline{\epsfig{file = 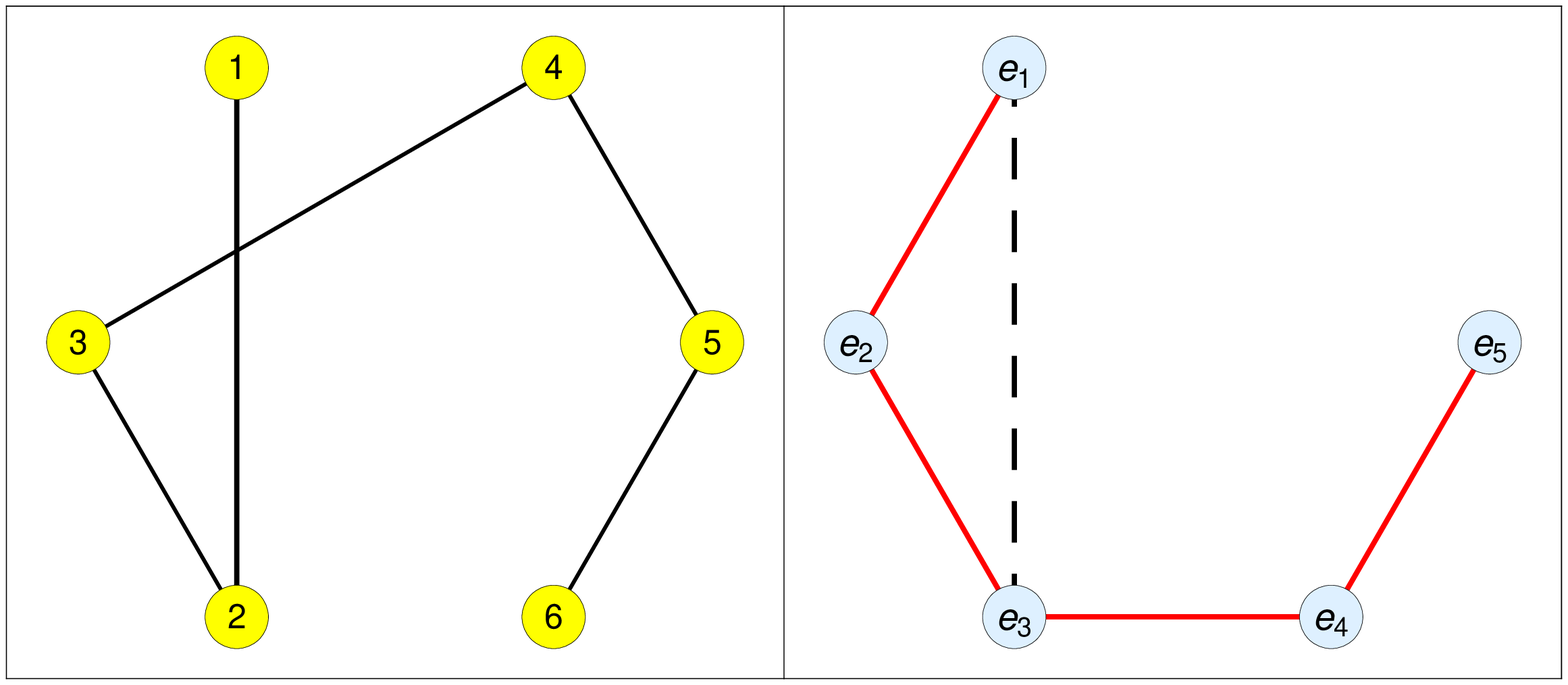,   width = .6\textwidth}} 
  \caption{$\P_6$ and its line\slash crossing graph}
  \label{fig:P6}\end{figure}

%
\subsection{Parameters}\label{subsec:tools}

We next  define parameters that help determine whether there is a geometric homomorphism between two geometric realizations of the same graph.  Proposition ~\ref{prop:observe} lists several properties that follow easily from  

the definition of these parameters.

\begin{figure}[htbp]
\centerline{\epsfig{file = 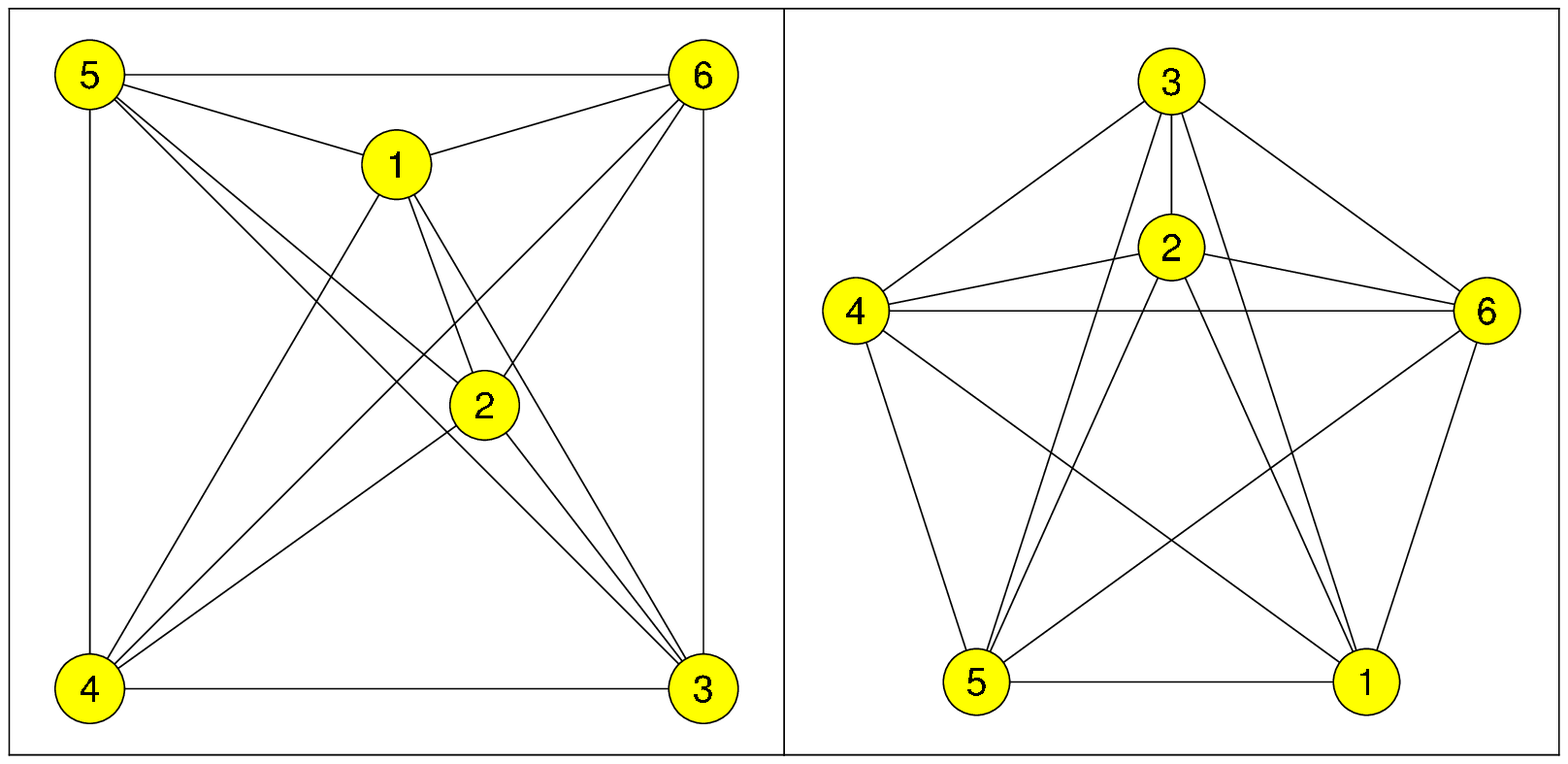,   width = .6\textwidth}} 
  \centering
$\overline K_6$\hspace{.3\textwidth}$\widehat K_6$
\caption{Vertex labels giving a geometric homomorphism from $\overline K_6$ to $\widehat K_6$}
  \label{fig:K6s}\end{figure}

Figure~\ref{fig:K6s} shows two geometric realizations, $\overline K_6$ and $\widehat K_6$,  of $K_6$.  The vertex labeling gives a geometric homomorphism from $\overline K_6$ to $\widehat K_6$. Below we define several parameters for geometric embeddings, which we then use to demonstrate that there is no geometric homomorphism from $\widehat K_6$ to $\overline K_6$.
In what follows we  let $e_{i,j}$  denote the edge from vertex $i$ to vertex $j$.

\definition If $\overline{G}$ is a geometric realization of a  graph $G$,  let $cr(\G)$ denote the total number of crossings in $\G$.   For $e\in E(\G)$ let  $cr(e)$ be the number of edges that cross the edge $e$ in $\G$. Let $E_0$  (resp., $E_\times$) denote the set of edges in $\overline{G}$ that have $cr(e) = 0$ \ (resp., $cr(e) > 0$).  If $|E_\times| = 0$ we say that $\overline{G}$ is a {\it plane realization} of $G$.  Let $G_0$ (resp., $G_\times $) denote the abstract graph that is the spanning subgraph of $G$ whose edge set is $E_0$ (resp., $E_\times$).  A clique in $\G$ is called a {\em convex clique} if its vertices are in convex position. The {\em convex clique number} of $\G$ is the maximum size of a convex clique in $\G$, denoted by $\widehat \omega(\G)$.

\begin{prop}\label{prop:observe}\rm
Let $\overline{G}$ and $ \widehat{G}$ be geometric realizations of a  graph $G$, and suppose $\overline{G}\stackrel{f} \preceq \widehat{G}$. Then each of the following conditions holds.
\begin{enumerate}
  \item\label{a}  $cr(\G) \leq cr(\hat G)$. 
  \item\label{b} For each $e\in E(\G)$, $cr(e) \leq cr(f(e))$.
  \item\label{c} $|E_0(\G)| \geq |E_0(\hat G)|$ and $|E_\times(\G)| \leq |E_\times(\hat G)|$.
  \item\label{d} $\hat G_0$ is a subgraph of $\G_0$.
  \item\label{e} $\widehat{\omega}(\G) \leq \widehat{\omega}(\widehat{G})$.
\end{enumerate}  
\end{prop}

\example In the graphs in Figure \ref{fig:K6s}, $cr(\overline K_6) = 8$, $cr(\widehat K_6)=11$,  $|E_0(\overline K_6)| = 7$,  $|E_0(\widehat K_6)| = 6$, $\widehat{\omega}(\overline K_6) = 4$ and $\widehat{\omega}(\widehat K_6)=5$. Thus parts \ref{a}, \ref{c},  and \ref{e} of Proposition~\ref{prop:observe} each imply that there is no geometric homomorphism from $\widehat K_6$ to $\overline K_6$.

\definition \rm Let $\G$ be a geometric realization of a graph $G$. For each $v\in V(\G)$, let $d_0(v)$ be the number of uncrossed edges incident to $v$ and  let $m(v)$ be the maximum number of times an edge that is incident to $v$ is crossed.

\begin{prop}\label{prop:obs2} \rm
If $\G \stackrel{f}{\preceq} \hat G$ is a geometric homomorphism,  then  for each $v\in V(\G)$, $d_0(v) \geq d_0(f(v))$ and $m(v) \leq m(f(v))$.
\end{prop}

\example In the example in Figure \ref{fig:K6s},  consider the vertex $v = 1$ in $\overline K_6$ and its image $f(v)$ in 
$\widehat K_6$. Then $d_0(v)=2$ and $m(v) = 2$,  while $d_0(f(v))=2$ and $m(f(v))=3$.

An effective way to use Proposition~\ref{prop:obs2} is to compare the values of each parameter over all vertices at once.  This motivates the following definitions.

\definition For a  geometric graph $\G$,  let    $D_0(\G)$ (resp., $M(\G)$) be the vector whose coordinates contain the values 
$\{d_0(v)\}_{v\in V(\G)}$ (resp., $\{m(v)\}_{v\in V(\G)}$),  listed in non-increasing order.


\definition\rm Given two vectors $\vec x$ and $\vec y$ in $\Z^n$, we say that $\vec x\leq \vec y$ if each coordinate of $\vec x$ has value that is at most the value in the corresponding coordinate of $\vec y$.  Let $X, Y$ be the vector of values of $\vec x$ and $\vec y$ listed in non-increasing order. 

\begin{lem}\label{lem:downarrow} \rm Let $\vec x, \vec y \in \mathbb{Z}^n$. If $\vec x \leq \vec y$, then  $X \leq Y$. \end{lem}

\begin{proof} 
By definition, the first coordinate of $X$ is $X_1=\max \big\{ x_i \mid 1 \leq i \leq n\big \}$,
and so $X_1  = x_{i_1}$ for some $i_1 \in \{1, \dots ,   n\}$.   But by assumption, $X_1 = x_{i_1} \leq y_{i_1} \leq \text{max} \big\{ y_i \mid 1 \leq i \leq n\big \} = Y_1.$

More generally, for all $2 \leq k \leq n$, the $k$-th entry of $X$ is less than or equal to at least $k$ entries of $\vec x$, say $x_{i_1},\ldots, x_{i_k}$.  Further by assumption, for each $h$,  $x_{i_h}\leq y_{i_h}$.  Since there are (at least) $k$ coordinates of $\vec y$ that have value at least $X_k$, the value $Y_k$ must be at least that large. Thus the $k$-th entry of $X$ is less than or equal to at least $k$ coordinates of $\vec Y$, and so  $X_k \leq Y_k$.  Thus $X \leq Y$.\end{proof}

\begin{cor}\rm \label{cor:vector} \rm If $\G \stackrel{f}{\preceq} \hat G$,  then:
\begin{enumerate}
\item\label{cor1} $D_0(\G) \geq D_0(\hat G)$; 
\item\label{cor2} $M(\G)\leq M(\hat G)$. 
\end{enumerate}
\end{cor} 

In  Figure \ref{fig:K6s},   $D_0(\overline K_6) = (3, 3, 3, 2, 2, 1) \geq D_0(\widehat K_6)=(3, 2, 2, 2, 2, 1)$, while $M(\overline K_6)=(4, 4, 3, 3, 2, 2) \leq M(\widehat K_6)=(4, 4, 3, 3, 3, 2)$.


\section{Posets for Geometric Paths}\label{sect:paths}

We now determine ${\mathcal P}_n$, the geometric homomorphism poset of the  path ${P}_n$ on $n$ vertices, for $n = 2, \ldots, 6$, and we state some properties of this poset for general $n$. Throughout this section, we denote the vertices of $P_n$ by $1, 2, \ldots, n$, and its edges by $e_i = \{i, i+1\}, i = 1, \ldots, n-1$.  
The following two lemmas are helpful in determining the geometric realizations of  $P_n$.

\begin{lem}\label{lem:p5crossingrule} \rm
If a geometric graph $\overline{G}$ contains $P_5$ as a subgraph, with vertices and edges numbered as above and $e_1 \times e_3$ and $e_2 \times e_4$ are both crossings in $\overline{G}$, then so is $e_1 \times e_4$. 
\end{lem}

\begin{proof}
Suppose $\overline{G}$ 
has both of the crossings $e_1 \times e_3$ and $e_2 \times e_4$. 
 Let $\ell$ be the line determined by edge $e_1$.
Since $e_1$ crosses $e_3$, we may  assume that vertex 3
lies above $\ell$ and vertex 4 lies below $\ell$, as indicated in Figure~\ref{fig:lemma3}. Let $C$ be the cone with vertex 4 and sides extending through vertices 1 and 2 (indicated by dashed lines in Figure~\ref{fig:lemma3}). 
For $e_3$ to cross $e_1$, both 3 and 
 $e_2$
 must be inside $C$. For $e_4$ to cross $e_2$, both 5 and 
 $e_4$
  must lie in the cone with vertex 4 and sides through 2 and 3, and 5 must also lie above $e_2$. This forces $e_4$ to cross $e_1$ in addition to crossing $e_2$.
\end{proof}

\begin{figure}[htbp]
 \centerline{\epsfig{file = 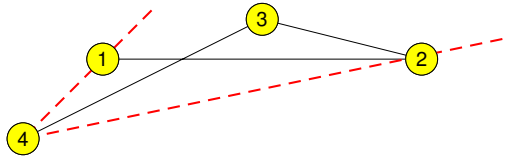,   width = .5\textwidth}} 
  \caption{For use in the proof of Lemma \ref{lem:p5crossingrule}}
  \label{fig:lemma3}
\end{figure}

 \begin{lem}\label{lem:pn_max} \rm A geometric realization of $P_n$ has at most ${(n-2)(n-3)}/{2}$ edge crossings.  Moreover, this bound is tight.\end{lem}

\begin{proof} The only possible crossing edge pairs in any geometric realization of $P_n$ are of the form $e_i \times e_j$ with $j-i\ge 2$; thus for any $i = 1, \ldots, n-3$,  there are $n-i-2$ higher-numbered edges that can cross $e_i$. A straightforward algebraic calculation shows that ${(n-2)(n-3)}/{2}$ is an upper bound on the number of edge crossings. Figure~\ref{fig:p7_maxcrossings} gives a 
geometric realization of $P_7$ that achieves this bound.
\end{proof}

\begin{figure}[htbp]
 \centerline{\epsfig{file = 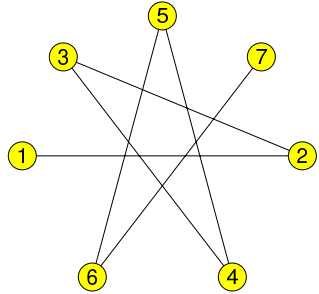,   width = .3\textwidth}} 
\caption{$\overline{P}_7$ with the maximum number of crossings}  \label{fig:p7_maxcrossings}
\end{figure}

To determine up to isomorphism all the possible geometric realizations of  $P_n$, $n = 2, \ldots, 5$,  first list all  sets whose elements are pairs of edges $e_i \times e_j$ with $j-i\ge 2$. 
Next  eliminate any sets that violate Lemma~\ref{lem:p5crossingrule}, and  identify any sets that are equivalent under an automorphism of $P_n$.  
Recall that the only two automorphisms of $P_n$ are the identity and the map that reverses the order of the vertices.
Finally, check that each of the remaining sets corresponds to a geometric realization.

To determine the structure of the geometric homomorphism poset, recall that by Proposition \ref{prop:preceqiff}, $\overline{P}_n \preceq \hatP_n$ if and only if there is  $\tilde f\in {\rm Aut}(L(P_n))$  that induces a graph homomorphism from  $EX(\overline{P}_n)$ to $EX(\hatP_n)$. The graph $L(P_n) = P_{n-1}$ has only the two automorphisms mentioned above.  Thus for each ordered pair of realizations, we need only check two automorphisms to see if they extend to color-preserving homomorphisms of the corresponding line\slash crossing graph.

For $n = 2, \ldots, 5$, all sets that satisfy Lemma~\ref{lem:p5crossingrule} have geometric realizations.
We state the poset results for $P_2, \ldots, P_5$ below.

\begin{thm} \label{thm:p1top5posets} \rm Let ${\mathcal P_n}$ be the poset of geometric realizations of $P_n$.

\begin{enumerate}
  \item Each of ${\mathcal P}_2$ and ${\mathcal P}_3$  is trivial, containing only the plane realization.
  \item ${\mathcal P}_4$ is a chain of two elements, in which the plane realization is the unique minimal element, and the realization with crossing $e_1\times e_3$ is the unique maximal element.
  \item ${\mathcal P}_5$ has the following five non-isomorphic geometric realizations and Hasse diagram as given in Figure \ref{fig:P5_poset}: $0.1 = \emptyset$ (the plane realization) ;  two 1-crossing realizations, $1.1 = \{e_1\times e_3\} \equiv \{e_2\times e_4\}$ and $1.2 = \{e_1\times e_4\}$;  a single 2-crossing realization, $2.1 = \{e_1\times e_3, e_1\times e_4\} \equiv \{e_1\times e_4, e_2\times e_4\}$; and a single 3-crossing realization, $3.1 = \{e_1\times e_3, e_1\times e_4, e_2\times e_4\}$. 
\end{enumerate}  
\end{thm}

\begin{proof}
It is straightforward to find geometric realizations with the given sets of crossings. These realizations, together with their line/crossing graphs (which aid in determining the poset relations), appear in Appendix~\ref{app:paths}. It follows from Lemmas~\ref{lem:p5crossingrule} and \ref{lem:pn_max} that there are no other realizations. 
\end{proof}

\begin{figure}[htbp]
\centerline{\epsfig{file = 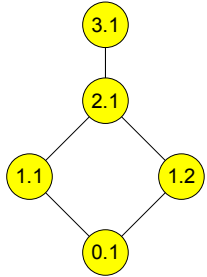,   width = .2\textwidth}}
  \centering
  \caption{The Hasse diagram for ${\mathcal P}_5$}
  \label{fig:P5_poset}
\end{figure}

For $P_6$, there is one set of crossing edge pairs that satisfies Lemma~\ref{lem:p5crossingrule}, but which does not correspond to a geometric realization of $P_6$, namely $\{e_1\times e_3, e_1\times e_4, e_1\times e_5, e_2\times e_5, e_3\times e_5\}$. The following lemma shows that this set can also be eliminated.

\begin{lem}\label{lem:p6rule} \rm
Suppose a geometric graph $\overline{G}$ contains $P_6$ as a subgraph, with vertices and edges numbered in the standard way. If $e_1\times e_3, e_1\times e_4, e_1\times e_5, e_2\times e_5$, and $e_3\times e_5$ are crossings in $\G$, then so is $e_2 \times e_4$.
\end{lem}

\begin{proof}
Since edges $e_1$ and $e_3$ cross, we may assume without loss of generality that $e_2$ is horizontal, and that $e_1$ and $e_3$ lie above $e_2$, as indicated in Figure~\ref{fig:claim2}. Since 
$\G$ contains the crossing $e_1\times e_4$,
vertex 5 is in one of the regions $T, E$, or $S$ shown  in Figure~\ref{fig:claim2}. 
If vertex 5 is in $E$ or $T$, then $e_5$ cannot cross all three of the edges $e_1, e_2,$ and $e_3$.
Thus  vertex 5 is in $S$, forcing the crossing $e_2 \times e_4$.
\end{proof}

\begin{figure}[htbp]
 \centerline{\epsfig{file = 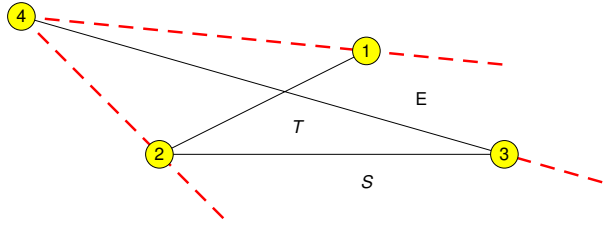,   width = .6\textwidth}}
  \caption{For use in the proof of Lemma~\ref{lem:p6rule}}
  \label{fig:claim2}
\end{figure}

We are now able to list all the non-isomorphic geometric realizations of $P_6$ and give the Hasse diagram for ${\mathcal P}_6$.

\begin{thm}\label{thm:p6geoms} \rm   The poset ${\mathcal P}_6$ has the following thirty-one non-isomorphic geometric realizations and has Hasse diagram as given in Figure \ref{fig:P6_poset}. 

 {\em 0 crossings:} 0.1 = $\emptyset$;

 {\em 1 crossing:} 
1.1 = $\{e_1 \times e_3\} \equiv \{e_3 \times e_5\}$,  
1.2 = $\{e_1 \times e_4\} \equiv \{e_2 \times e_5\}$,  
1.3 = $\{e_1 \times e_5\}$,  
1.4 = $\{e_2 \times e_4\}$;

{\em 2 crossings:} 
2.1 = $\{e_1 \times e_3, e_1 \times e_4\}$, 
2.2 =  $\{e_1 \times e_3, e_1 \times e_5\}$,
2.3 = $\{e_1 \times e_3, e_2 \times e_5\}$,
2.4 = $\{e_1 \times e_3, e_3 \times e_5\}$,  
2.5 = $\{e_1 \times e_4, e_1 \times e_5\}$,
2.6 = $\{e_1 \times e_4, e_2 \times e_4\}$,
2.7 = $\{e_1 \times e_4, e_2 \times e_5\}$,  
2.8 = $\{e_1 \times e_5, e_2 \times e_4\}$;

 {\em 3 crossings:} 
3.1 = $\{e_1 \times e_3, e_1 \times e_4, e_1 \times e_5\}$, 
3.2 = $\{e_1 \times e_3, e_1 \times e_4, e_2 \times e_4\}$,
3.3 = $\{e_1 \times e_3, e_1 \times e_4, e_2 \times e_5\}$,
3.4 = $\{e_1 \times e_3, e_1 \times e_4, e_3 \times e_5\}$,
3.5 = $\{e_1 \times e_3, e_1 \times e_5, e_2 \times e_5\}$,
3.6 = $\{e_1 \times e_3, e_1 \times e_5, e_3 \times e_5\}$,  
3.7 = $\{e_1 \times e_4, e_1 \times e_5, e_2 \times e_4\}$ 
3.8 = $\{e_1 \times e_4, e_1 \times e_5, e_2 \times e_5\}$,  
3.9 = $\{e_1 \times e_4, e_2 \times e_4, e_2 \times e_5\}$;

 {\em 4 crossings:} 
4.1 = $\{e_1 \times e_3, e_1 \times e_4, e_1 \times e_5, e_2 \times e_4\} 
 \equiv \{e_1 \times e_5, e_2 \times e_4, e_2 \times e_5, e_3 \times e_5\}$,  
4.2 = $\{e_1 \times e_3, e_1 \times e_4, e_1 \times e_5, e_2 \times e_5\}
 \equiv \{e_1 \times e_4, e_1 \times e_5, e_2 \times e_5, e_3 \times e_5\}$,  
4.3 = $\{e_1 \times e_3, e_1 \times e_4, e_1 \times e_5, e_3 \times e_5\}
 \equiv \{e_1 \times e_3, e_1 \times e_5, e_2 \times e_5, e_3 \times e_5\}$,  
4.4 = $\{e_1 \times e_3, e_1 \times e_4, e_2 \times e_4, e_2 \times e_5\}
 \equiv \{e_1 \times e_4, e_2 \times e_4, e_2 \times e_5, e_3 \times e_5\}$,  
4.5 = $\{e_1 \times e_3, e_1 \times e_4, e_2 \times e_5, e_3 \times e_5\}$,  
4.6 = $\{e_1 \times e_4, e_1 \times e_5, e_2 \times e_4, e_2 \times e_5\}$;

 {\em 5 crossings:} 
5.1 = $\{e_1 \times e_3, e_1 \times e_4, e_1 \times e_5, e_2 \times e_4, e_2 \times e_5\} 
 \equiv \{e_1 \times e_4, e_1 \times e_5, e_2 \times e_4, e_2 \times e_5, e_3 \times e_5\}$,  
5.2 = $\{e_1 \times e_3, e_1 \times e_4, e_2 \times e_4, e_2 \times e_5, e_3 \times e_5\}$;

 {\em 6 crossings:} 
6.1 = $\{e_1 \times e_3, e_1 \times e_4, e_1 \times e_5, e_2 \times e_4, e_2 \times e_5, e_3 \times e_5\}$.
\end{thm}

\begin{proof}
It is straightforward to find geometric realizations with the given sets of crossings. These realizations, together with their line/crossing graphs (which aid in determining the poset relations), appear in Appendix~\ref{app:paths}.
It follows from Lemmas~\ref{lem:p5crossingrule}, \ref{lem:pn_max}, and \ref{lem:p6rule} that there are no others. 
\end{proof}

The following theorem lists some  properties of ${\mathcal P}_n$ for $n \ge 3$.

\begin{thm}\label{thm:pn_props} \rm For $n \ge 3$, ${\mathcal P}_n$ has the following properties.
\begin{enumerate}
  \item\label{a1} There is a unique minimal element, corresponding to the plane realization of $P_n$.
  \item\label{a2} There is a unique maximal element, corresponding to the realization of $P_n$ with  ${(n-2)(n-3)}/{2}$ crossings.
  \item\label{a3} ${\mathcal P}_n$ has a chain of size ${(n-2)(n-3)}/{2}+ 1$. In particular, for each $c$  with  $0 \le c \le {(n-2)(n-3)}/{2}$, there is at least one realization of $P_n$ with exactly $c$ crossings.
  \item\label{a4} For $1 \le k \le n$,  ${\mathcal P}_k$ is  isomorphic to a sub-poset of ${\mathcal P}_n$.
\end{enumerate}
\end{thm}

\begin{proof}
Properties~\ref{a1} and \ref{a2} are easily seen to be true. For Property~\ref{a3}, consider a geometric realization  $\P_n$ of $P_n$ with $c \ge 1$ crossings. Such a realization can be modified to create a new realization $\hatP_n$ with $c-1$ crossings, and with $\hatP_n \prec \P_n$, by
sliding the vertex $n$ along edge $e_{n-1}$ until it passes over a crossing edge, and then erasing the section of $e_{n-1}$ that extends beyond this point. If $e_{n-1}$ has no crossings, we slide vertices $n$ and $n-1$ along edge $e_{n-2}$, erasing what remains of $e_{n-1}$ and $e_{n-2}$, and so on.  
We can continue in a similar manner to remove one crossing at a time until there are none left; the process is illustrated in Figure~\ref{fig:P7_embed}. Since Lemma~\ref{lem:pn_max} guarantees that $P_n$ has a realization with  ${(n-2)(n-3)}/{2}$ crossings,  Property~\ref{a3} follows.

For Property~\ref{a4}, suppose we have some geometric realization of $P_k$. We can replace the uncrossed segment of edge $e_{k-1}$ nearest to vertex $k$ with a path from $k$ to $n$ to obtain a geometric realization of $P_n$ with the same crossings. By doing this for each realization of $P_k$, we see that its poset of geometric realizations is isomorphic to a sub-poset of the poset of geometric realizations of $P_n$.
\end{proof}

A {\em cover} of an element $x$ in a poset $\mathcal{P}$ is an element $y \in \mathcal{P}$ such that $x \prec y$ and  no $z \in \mathcal{P}$ satisfies $x \prec z \prec y$. $\mathcal{P}$ is called a  {\it graded} poset if there is  a {\em rank} function $\rho:\mathcal{P}\to\mathbb{N}$ such that for all $x, y \in \mathcal{P}$, 1) all minimal elements have the same value under the rank function, 2)  if $x \prec y$, then $\rho(x) < \rho(y)$, and 3) if $y$ covers $x$, then $\rho(y) = \rho(x) + 1$.  Note that if a poset is graded then all maximal chains between a given pair of elements must have the same length.

A reasonable conjecture for geometric homomorphism posets is that the number of edge crossings acts as a rank function.  Condition 1 holds by Proposition~\ref{prop:observe}.  However, Condition 2 fails to hold in exactly one instance in $\mathcal{P}_6$:  realization 6.1 covers realization 4.3, yet it has two more crossings.  In fact, the poset $\mathcal{P}_6$ does not admit any rank function, because it has maximal chains between 0.1 and 6.1 which have different lengths: $0.1 \prec 1.4 \prec 2.8 \prec 3.7 \prec 4.6 \prec 5.1 \prec 6.1$ and $0.1 \prec 1.1 \prec 2.2 \prec 3.5 \prec 4.3 \prec 6.1$. Hence, $\mathcal{P}_6$ is not a graded poset. It follows from Property~\ref{a4} that $\mathcal{P}_n$ is not a graded poset for any $n \geq 6$.

 \begin{figure}[htbp]
 \centerline{\epsfig{file = 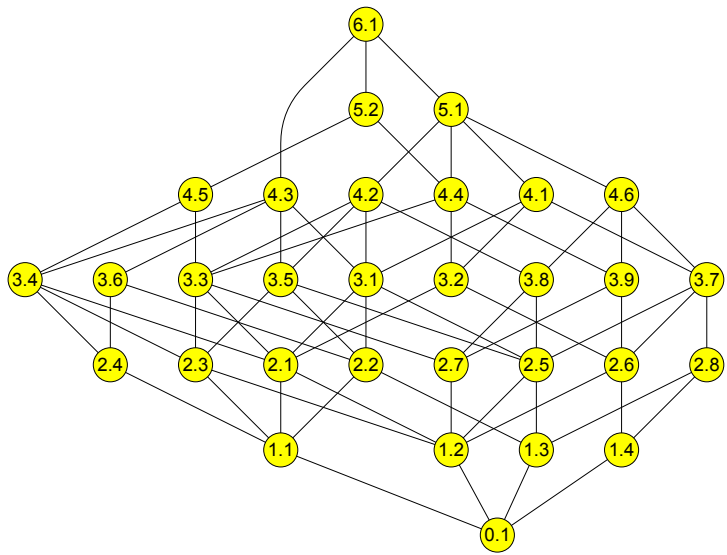,   width = .7\textwidth}}
  \caption{The Hasse diagram for ${\mathcal P}_6$}
  \label{fig:P6_poset}
\end{figure}

A {\it lattice} is a poset in which any two elements have a unique supremum (join) and unique infimum (meet).  Figure \ref{fig:P6_poset} shows that $\mathcal{P}_6$ is not a lattice because (for example) realizations 3.5 and 3.1 have both 4.3 and 4.2 as suprema, and realizations 4.3 and 4.2 have both 3.5 and 3.1 as infima.  It follows from Property~\ref{a4} that $\mathcal{P}_n$ is not a lattice for any $n \geq 6$.

\begin{figure}[htbp]
 \centerline{\epsfig{file = 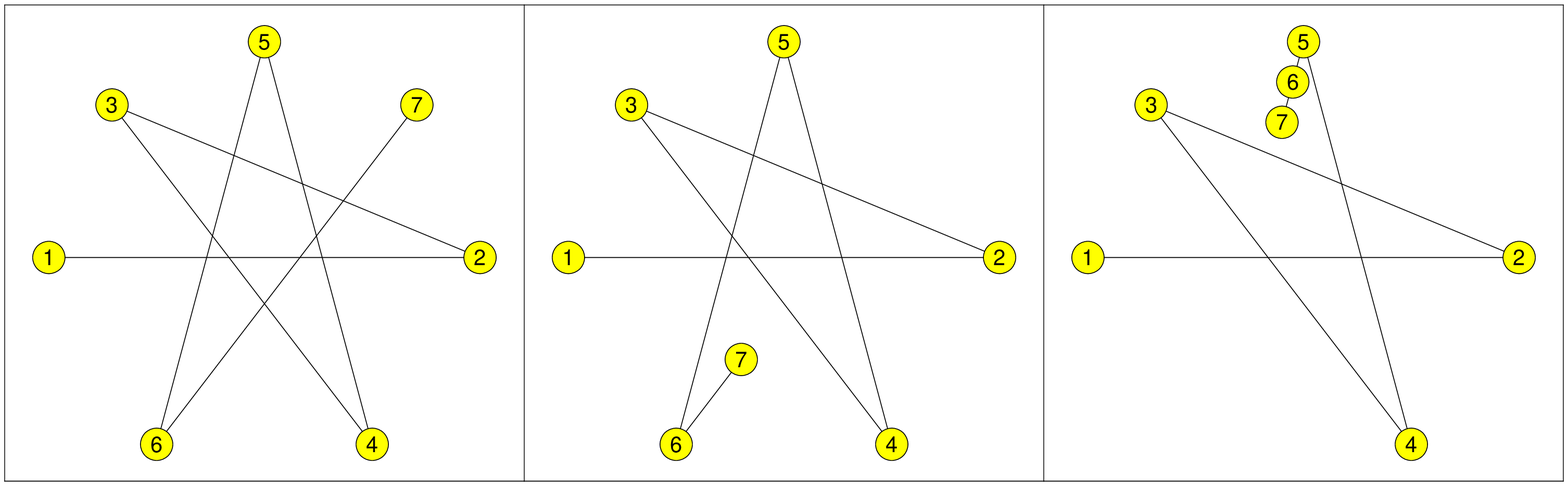,   width = .9\textwidth}}
  \caption{Geometric realizations of $P_7$ with 10 crossings,  6 crossings,  and 3 crossings}
  \label{fig:P7_embed}
\end{figure}

\section{Posets for Geometric Cycles}\label{sect:cycles}

We now determine ${\mathcal C}_n$, the geometric homomorphism poset of the path ${C}_n$ on $n$ vertices, for $n = 3, \ldots, 6$, and we state some properties of this poset for general $n$. Throughout this section we denote the vertices of $C_n$ by $1, 2, \ldots, n$, and its edges by $e_i = \{i, i+1\}, i = 1, \ldots, n-1$, and $e_n = \{n, 1\}$.   

The maximum number of crossings in a geometric realization of $C_n$ was  determined in 1977 by Furry and Kleitman~\cite{Furry77}; their results are summarized in the next lemma.

\begin{lem} \rm {\cite{Furry77}}\label{lem:furry}
For $n \geq 3$, a geometric realization of $C_n$ has at most ${n(n-3)}/{2}$ edge crossings if $n$ is odd and ${n(n-4)}/{2} + 1$ edge crossings if $n$ is even. Moreover, these bounds are tight.
\end{lem}

Figure~\ref{fig:c10_max} shows such a realization for $n=10$.

\begin{figure}[htbp]
 \centerline{\epsfig{file = 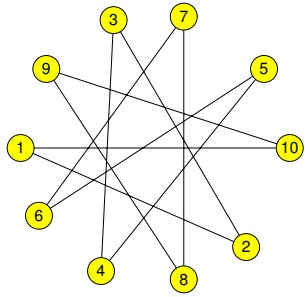,   width = .3\textwidth}}
  \caption{A realization of $C_{10}$ with the maximum number of crossings}
  \label{fig:c10_max}
\end{figure}

The techniques of the previous section can be used to find the elements of $\mathcal{C}_n$.  For $3 \le n \le 5$, every set of crossing edge pairs that satisfies Lemma~\ref{lem:p5crossingrule} corresponds to a geometric realization.
For $C_6$, this is the case for all 
geometric realizations with at most two crossings.  For realizations with three or more crossings, some cases  require additional lemmas.

\begin{lem}\label{lem:insideoutside} \rm
Suppose $\overline{C}_n$ is a 
geometric realization of  
the cycle $C_n$ that has crossings $e_i \times e_k$ and $e_j \times e_\ell$, where $i < j < k < \ell$.  Then there is at least one additional crossing $e_\alpha \times e_\beta$ where $i \le \alpha  \le k$ and $k \le \beta \le i$ (mod $n$) (and  $\{\alpha, \beta\} \ne \{i, k\}$).
\end{lem}

\begin{proof}
Suppose there is no such additional crossing. Place a vertex $v$ at the crossing $e_i \times e_k$, subdividing each of those two edges. Starting at edge 
$\{v, i+1\}$ of the modified graph, follow the cycle in order of increasing vertex number, coloring each edge red, until the crossing $e_i \times e_k$ is reached again at edge 
$\{k, v\}$.
Then follow the rest of the cycle, beginning at  edge 
$\{v, k+1\}$,
coloring each edge blue, until the final edge 
$\{i, v\}$ is reached. Note that edge $e_j$ is red and edge $e_\ell$ is blue. 
The red cycle is a closed, but not necessarily simple, rectilinear curve in the plane.
From the hypotheses of the lemma and our assumption that the conclusion is false, this red curve does not cross either of the edges 
$\{i, v\}$ and $\{v, k+1\}$,
so these two edges lie in the same region of the plane determined by the red curve. If we now follow the blue curve starting at $k+1$, then the red-blue crossing $e_j \times e_\ell$ takes the blue curve into a different region of the plane determined by the red curve. But since we have assumed that the additional crossing of the lemma does not exist, the blue curve cannot return to end at vertex $i$, which is a contradiction. 
\end{proof}

Lemma~\ref{lem:c6crossingrules} is a multi-part technical lemma. 
We prove the first part below; the proofs of the others, which are similar to the proof of Lemma~\ref{lem:p6rule}, appear in Appendix~\ref{app:cycles}.

\begin{lem}\label{lem:c6crossingrules} \rm
Let $\overline{C}_6$ be a geometric realization of the cycle $C_6$,  with edges labeled consecutively,  $e_1,  e_2,  \ldots,  e_6$.
\begin{enumerate}
  \item\label{cl1} If $\overline{C}_6$ contains the crossings $e_1\times e_3,  e_1\times e_4$,  and $e_1\times e_5$,  then it doesn't contain the crossing $e_2\times e_6$.
  \item\label{cl2} If $\overline{C}_6$ contains the crossings $e_1\times e_3,  e_1\times e_4,  e_1\times e_5,  e_2\times e_4$,  and $e_4\times e_6$,  then it also contains the crossing $e_2\times e_5$.
  \item\label{cl3} If $\overline{C}_6$ contains the crossings $e_1\times e_3,  e_1\times e_4,  e_2\times e_4$,  and $e_2\times e_5$,  then it also contains at least one of  the crossings $e_1\times e_5,  e_3\times e_5,  e_3\times e_6$.
  \item\label{cl4} If $\overline{C}_6$ contains the crossings $e_1\times e_3,  e_1\times e_4,  e_2\times e_5$,  and $e_4\times e_6$,  then it also contains at least one of  the crossings $e_2\times e_4,   e_3\times e_5,  e_3\times e_6$.
  \item\label{cl5} If $\overline{C}_6$ contains the crossings $e_1\times e_3,  e_1\times e_4,  e_2\times e_5$,  and $e_3\times e_6$,  then it also contains at least one of  the crossings $e_2\times e_4,  e_2\times e_6,  e_3\times e_5,  e_4\times e_6$.
\end{enumerate}  
\end{lem}

\begin{proof} 
For part~\ref{cl1},  suppose
 $\overline{C}_6$ contains the crossings $e_1\times e_3,  e_1\times e_4$,  and $e_1\times e_5$. Let $h$ be the  line through edge $e_1$. Since $e_1$ crosses every edge of the path joining vertices 3 and 6,  the vertices 3 and 6 lie on opposite sides of $h$. Thus the edges $e_2=\{2,  3\}$ and $e_6=\{6,  1\}$ also lie on opposite sides of $h$ and so do not cross. 
 \end{proof}

Part~\ref{cl1}
and its proof generalize to give us the following corollary.

\begin{cor}\label{lem:cncrossingrule} \rm
Let $\overline{C}_n$ be a geometric realization of the cycle $C_n$, 
where $n\ge 4$ is even. If $\overline{C}_n$ contains the crossings $e_1\times e_3, e_1\times e_4, \ldots, e_1\times e_{n-1}$, then it doesn't contain the crossing $e_2\times e_n$.
\end{cor}

To determine the the elements of the poset ${\mathcal C}_6$,   look at all possible sets of crossing edge pairs,  and delete sets that don't satisfy Lemmas~\ref{lem:p5crossingrule}, \ref{lem:p6rule}, \ref{lem:insideoutside} or~\ref{lem:c6crossingrules}.

Next,  identify those that are equivalent under an automorphism of $C_n$. There are  $2n$ such automorphisms: each of the rotations and each of these composed with the reflection map. To determine the geometric homomorphisms among the remaining sets, recall that by Proposition~\ref{prop:preceqiff}, it suffices to look for automorphisms of the line graph $L(C_n)$ that extend to homomorphisms on the edge crossing graphs. Since $L(C_n) = C_n$, these automorphisms are precisely the $2n$ automorphisms mentioned above.

Theorem~\ref{thm:c1toc6posets} lists the elements of the poset of geometric realizations of $C_n$ for $3 \le n \le 6$;  all nontrivial realizations are given up to isomorphism.

\begin{thm} \label{thm:c1toc6posets} \rm  Let ${\mathcal C_n}$ be the poset of geometric realizations of $C_n$.
\begin{enumerate}
  \item ${\mathcal C}_3$  is trivial, containing only the plane realization.
  \item ${\mathcal C}_4$ is a chain of two elements, in which the plane realization is the unique minimal element, and the realization with crossing $e_1\times e_3$ is the unique maximal element.
\item ${\mathcal C}_5$ is a chain of five elements: the plane realization $0.1 = \emptyset$, $1.1 = \{e_1\times e_3\}$, $2.1 = \{e_1\times e_3, e_1\times e_4\}$, $3.1 = \{e_1\times e_3, e_1\times e_4, e_2\times e_4\}$, and $5.1 = \{e_1\times e_3, e_1\times e_4, e_2\times e_4, e_2\times e_5, e_3\times e_4\}$.
 \item ${\mathcal C}_6$ has the following twenty-six non-isomorphic geometric realizations and has Hasse diagram as given in Figure~\ref{fig:C6_poset}:

 {\em 0 crossings:}  $0.1 = \emptyset$ (the plane realization);  

 {\em 1 crossing:} 
$1.1 = \{e_1\times e_3\}$,  
$1.2 = \{e_1\times e_4\}$;  

 {\em 2 crossings:} 
$2.1 = \{e_1\times e_3, e_1\times e_4\}, 
2.2 = \{e_1\times e_3, e_1\times e_5\}$,  
$2.3 = \{e_1\times e_3, e_4\times e_6\}$; 

 {\em 3 crossings:}  
$3.1 = \{e_1\times e_3, e_1\times e_4, e_1\times e_5\}, 
3.2 = \{e_1\times e_3, e_1\times e_4, e_2\times e_4\}, 
3.3 = \{e_1\times e_3, e_1\times e_4, e_2\times e_5\}, 
3.4 = \{e_1\times e_3, e_1\times e_4, e_3\times e_5\}, 
3.5 = \{e_1\times e_3, e_1\times e_4, e_3\times e_6\}, 
3.6 = \{e_1\times e_3, e_1\times e_4, e_4\times e_6\}, 
3.7 = \{e_1\times e_3, e_1\times e_5, e_3\times e_5\}, 
3.8 = \{e_1\times e_4, e_2\times e_5, e_3\times e_6\}$;

 {\em 4 crossings:} 
$4.1 = \{e_1\times e_3, e_1\times e_4, e_1\times e_5, e_2\times e_4\}, 
4.2 = \{e_1\times e_3, e_1\times e_4, e_1\times e_5, e_2\times e_5\}, 
4.3 = \{e_1\times e_3, e_1\times e_4, e_1\times e_5, e_3\times e_5\}, 
4.4 = \{e_1\times e_3, e_1\times e_4, e_2\times e_5, e_3\times e_5\}, 
4.5 = \{e_1\times e_3, e_1\times e_4, e_3\times e_6, e_4\times e_6\}$;

 {\em 5 crossings:}  
 $5.1 = \{e_1\times e_3, e_1\times e_4, e_1\times e_5, e_2\times e_4, e_2\times e_5\},  
5.2 = \{e_1\times e_3, e_1\times e_4, e_2\times e_4, e_2\times e_5, e_3\times e_5\},  
5.3 = \{e_1\times e_3, e_1\times e_4, e_2\times e_4, e_2\times e_5, e_3\times e_6\},  
5.4 = \{e_1\times e_3, e_1\times e_4, e_2\times e_5, e_3\times e_6, e_4\times e_6\}$;

 {\em 6 crossings:} 
$6.1 = \{e_1\times e_3, e_1\times e_4, e_1\times e_5, e_2\times e_4, e_2\times e_5, e_3\times e_5\}$,  
$6.2 =   \{e_1\times e_3, e_1\times e_4, e_1\times e_5, e_2\times e_4, e_2\times e_5, e_3\times e_6\}$;

 {\em 7 crossings:} 
$7.1 = \{e_1\times e_3, e_1\times e_4, e_1\times e_5, e_2\times e_4, e_2\times e_5, e_3\times e_6, e_4\times e_6\}$.

\end{enumerate}  
\end{thm}

\begin{proof}
It is straightforward to find geometric realizations with the given sets of crossings. These realizations, together with their line/crossing graphs (which aid in determining the poset relations) appear in Appendix~\ref{app:cycles}.  It follows from Lemmas~\ref{lem:p5crossingrule}, \ref{lem:p6rule}, \ref{lem:furry}, \ref{lem:insideoutside} and \ref{lem:c6crossingrules} that there are no other realizations.
\end{proof}

\begin{figure}[htbp]
 \centerline{\epsfig{file = 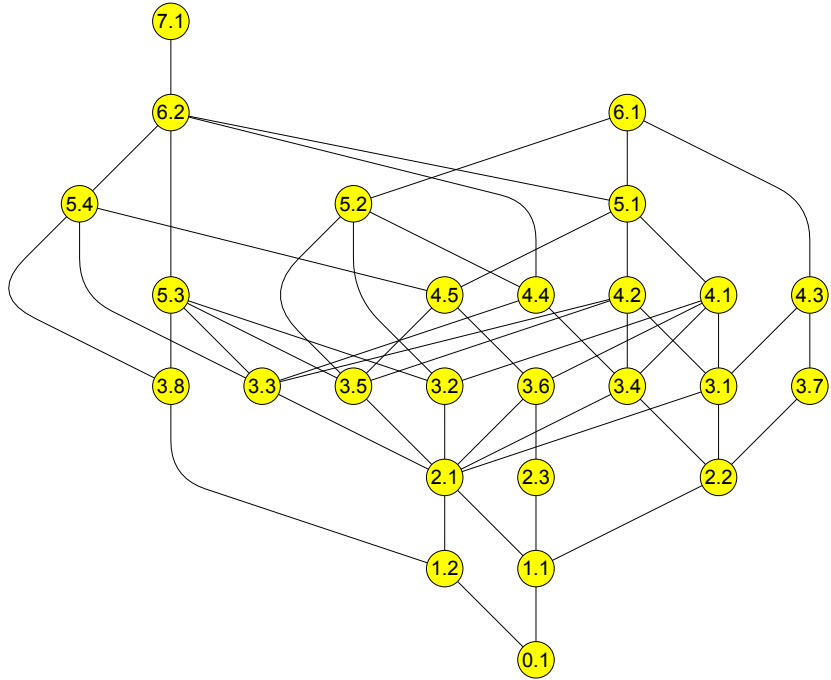,   width = .8\textwidth}}
  \caption{The Hasse diagram for ${\mathcal C}_6$}
  \label{fig:C6_poset}
\end{figure}

 Unlike the case of ${\mathcal P}_n$, we see from the geometric realizations of $C_5$ that not every possible number of crossings up to the maximum is necessarily achieved. On the other hand, $C_6$ has at least one realization with each number of crossings from 0 up to its maximum of 7. Furry and Kleitman have shown that this is representative of the geometric realizations of all odd and even cycles. This is stated in  Theorem \ref{FurryKleitman} below.

\begin{thm}\label{FurryKleitman} \rm \cite{Furry77}
For $n$ even, $n\ge4$, $\overline{C}_n$ can have any number of crossings from 0 up to the maximum of $n(n-4)/2+1$. For $n$ odd, $n\ge3$, $\overline{C}_n$ can have any number of crossings from 0 up to the maximum of $n(n-3)/2$, except there is no geometric realization with  $n(n-3)/2-1$ crossings.
\end{thm}

Finally we mention two more properties that are true in general for any geometric realization of the cycle $C_n$. The first one is obvious, and the second is easy to see, since any realization of $C_n$ can be replaced by one of $C_{n+1}$, by subdividing the edge $e_n$ into two edges, $e_n$ and $e_{n+1}$, so that the new edge $e_{n+1}$ has no crossings.

\begin{thm}\label{thm:cn_props} \rm
For $n \ge 3$, the poset ${\mathcal C}_n$ has the following properties.
\begin{enumerate}
  \item\label{b1} There is a unique minimal element, corresponding to the plane realization of $C_n$.
  \item\label{b2} There is a unique  maximal element, corresponding to the geometric realization with the maximum number of crossings, as given in Theorem~\ref{FurryKleitman}.
  \item\label{b3} For $3 \le k \le n$, the poset ${\mathcal C}_k$ is isomorphic to a sub-poset of ${\mathcal C}_n$.
\end{enumerate}
\end{thm}

Note that $\mathcal{C}_6$ is not a graded poset because it has maximal chains between 0.1 and 6.1 which have different lengths:
$0.1 \prec 1.1 \prec 2.2 \prec 3.7 \prec 4.3 \prec 6.1$ and $0.1 \prec 1.1 \prec 2.2 \prec 3.1 \prec 4.1 \prec 5.1 \prec 6.1$.
Thus by Property \ref{b2}, for all $n\geq 6$, $\mathcal{C}_n$ is not a graded poset. 
Also,  $\mathcal{C}_6$ is not a lattice because, for example, realizations 2.1 and 2.2 do not have a unique supremum. By Property \ref{b2},  $\mathcal{C}_n$ is not a lattice for all $n\geq 6$.

\section{Posets for Geometric Cliques}\label{sect:cliques}

We now determine ${\mathcal K}_n$, the geometric homomorphism poset of the clique ${K}_n$  for $n = 3, \ldots, 6$, and we state some properties of this poset for general $n$. Throughout this section we denote the vertices of $K_n$ by $1, 2, \ldots, n$, and its edges by $e_{ij} = \{i, j\}, i \ne j \in \{1, \ldots, n\}$.  

In \cite{HT}, Harborth and Th\"{u}rmann give all non-isomorphic geometric realizations of $K_n$ for $3\leq n \leq 6$.  Recall that their definition for geometric isomorphism is stricter than the definition being used here.  However, that only means that 
in general, 
our set of non-isomorphic geometric realizations may be smaller than theirs.  That is, two geometric realizations that Harborth and Th\"{u}rmann consider non-isomorphic, we may consider isomorphic.  However, in the cases $K_3, K_4, K_5, K_6$, 
all pairs that are non-isomorphic according to Harborth are also non-isomorphic according to us.

\begin{thm} \label{thm:k1tok6posets} \rm Let ${\mathcal K_n}$ be the poset of geometric realizations of $K_n$.

\begin{enumerate}
  \item ${\mathcal K}_3$  is trivial, containing only the plane realization.
  
  \item ${\mathcal K}_4$ is a chain of two elements, in which the plane realization is the unique minimal element, and the realization with crossing $e_{1,3}\times e_{2,4}$ is the unique maximal element.
  
\item ${\mathcal K}_5$ is  a chain of three elements: $1.1=\{e_{3,5}\times e_{2,4}\}$, $3.1=\{e_{1,4}\times e_{2,5}$, $e_{1,4}\times e_{3,5}$, $e_{2,4}\times e_{3,5}, \}$, and $5.1 = \{e_{1,3}\times e_{2,4}$, $e_{1,3}\times e_{2,5}$, $e_{1,4}\times e_{2,5}$, $e_{1,4}\times e_{3,5}$, $e_{2,4}\times e_{3,5}\}$.

 \item ${\mathcal K}_6$ has Hasse diagram as given in Figure~\ref{fig:HasseK6poset} with fifteen  non-isomorphic geometric realizations:
 
 {\em 3 crossings:}   $3.1 = \{e_{1, 3}\times e_{2, 6},  e_{1, 4}\times e_{2, 5},  e_{3, 5}\times e_{4, 6}\}$; 
 
 {\em 4 crossings:}  $4.1 = \{e_{1, 3}\times e_{2, 6},  e_{1, 4}\times e_{3, 5},  e_{1, 4}\times e_{5, 6},  e_{3, 5}\times e_{4, 6}\}$; 

 {\em 5 crossings:}  
 $5.1 = \{e_{1, 3}\times e_{2,4},  e_{1, 3}\times e_{2, 6},  e_{1, 4}\times e_{2, 6},  e_{1, 4}\times e_{3, 6},  e_{2, 4}\times e_{3, 6}\}$,  
$5.2 =\{e_{1, 3}\times e_{2, 6},  e_{1, 4}\times e_{2, 6},  e_{1, 4}\times e_{3,5},   e_{1, 4}\times e_{3,6},  e_{3, 5} \times e_{4, 6}  \}$; 

 {\em 6 crossings:} 
 $6.1 = \{e_{1, 4}\times e_{2, 5},  e_{1, 4}\times e_{2, 6},  e_{1, 4}\times e_{3, 6},  e_{2, 4}\times e_{3, 6},  e_{2, 5}\times e_{3, 6},  e_{2, 5}\times e_{4, 6}   \}$; 
 
 {\em 7 crossings:} 
 $7.1 = \{e_{1, 3}\times e_{2, 5},  e_{1, 3}\times e_{2, 6},  e_{1, 4}\times e_{2, 5},  e_{1, 4}\times e_{3, 5},    e_{1, 4}\times e_{5, 6},  e_{1, 6}\times e_{2, 5}, e_{3, 5}\times e_{4, 6} \}$, 
 $7.2 =\{e_{1, 3} \times e_{2, 6},  e_{1, 4} \times e_{2, 5},  e_{1, 4} \times e_{3, 5},  e_{1, 5} \times e_{4, 6},  e_{2, 4} \times e_{3, 5},  e_{2, 5} \times e_{4, 6},  e_{3, 5} \times e_{4, 6} \}$; 
 
 {\em 8 crossings:} 
$8.1 = \{e_{1, 3}\times e_{2, 4},  e_{1, 3}\times e_{2, 6},  e_{1, 4}\times e_{3, 5},  e_{1, 4} \times e_{5, 6},  e_{1, 6}\times e_{2,4},  e_{2, 4} \times e_{3, 5},  e_{2, 4} \times e_{5, 6},  e_{3, 5} \times e_{4, 6}\}$, 
$8.2 = \{e_{1, 3} \times e_{2,4},  e_{1, 3} \times e_{2, 6},  e_{1, 4} \times e_{2, 6},  e_{1, 4} \times e_{3, 6},  e_{1, 5} \times e_{2, 6},  e_{1, 5} \times e_{3,6},  e_{1, 5} \times e_{4, 6},  e_{2, 4} \times e_{3, 6}\}$; 
 
 {\em 9 crossings:} 
 $9.1=\{e_{1, 3} \times e_{2, 5},   e_{1, 3} \times e_{2, 6},  e_{1, 4} \times e_{2, 5},  e_{1, 4} \times e_{2, 6},  e_{1, 4} \times e_{3, 5},  e_{1, 4} \times e_{3, 6},  e_{2, 5} \times e_{3, 6},  e_{2, 5} \times e_{4, 6},  e_{3, 5} \times e_{4, 6}  \}$, 
 $9.2 = \{e_{1, 3} \times e_{2, 4},  e_{1, 3} \times e_{2, 5},  e_{1, 3} \times e_{2, 6},  e_{1, 4} \times e_{2, 5},  e_{1, 4} \times e_{2, 6}, e_{1, 4} \times  e_{3, 6},  e_{2, 4} \times e_{3, 6},  e_{2, 5} \times e_{3, 6},  e_{2, 5} \times e_{4, 6}\}$; 
 
 {\em 10 crossings:} 
 $10.1 = \{e_{1, 3} \times e_{2, 4},  e_{1, 3} \times e_{2, 5},  e_{1, 3} \times e_{2, 6},  e_{1, 4} \times e_{2, 5},  e_{1, 4} \times e_{3, 5},  e_{1, 4} \times e_{5, 6},  e_{1, 6} \times e_{2, 5},  e_{2, 4} \times e_{3, 5},  e_{2, 4} \times e_{3, 6},  e_{3, 5} \times e_{4, 6}\}$; 
  
 {\em 11 crossings:} 
 $11.1 =\{e_{1, 3} \times e_{2,4},  e_{1, 3} \times e_{2,5},  e_{1, 3}  \times e_{2, 6},  e_{1, 4} \times e_{2, 5},  e_{1, 4} \times e_{3, 5},  e_{1, 4} \times e_{5, 6}, e_{1, 6} \times e_{2, 4},  e_{1, 6} \times e_{2, 5},    e_{2, 4}  \times e_{3, 5},  e_{2, 4} \times e_{5, 6},  e_{3, 5} \times e_{4, 6}  \}$; 
 
 {\em 12 crossings:} 
 $12.1 =\{e_{1, 3} \times e_{2, 4},  e_{1, 3} \times e_{2, 5},  e_{1, 3} \times e_{2,6},  e_{1, 4} \times e_{2, 5},  e_{1, 4} \times e_{2, 6},  e_{1, 4} \times e_{3, 5},  e_{1, 4} \times e_{3, 6},  e_{2, 4} \times e_{3, 5},  e_{2, 4} \times e_{3, 6},  e_{2, 5} \times e_{3, 6},  e_{2, 5} \times e_{4, 6},  e_{3, 5} \times e_{4, 6}\}$; 
 
 {\em 15 crossings:} 
 $15.1 =\{e_{1, 3} \times e_{2, 4},  e_{1, 3} \times e_{2, 5},  e_{1, 3} \times e_{2, 6},  e_{1, 4} \times e_{2, 5},  e_{1, 4} \times e_{2, 6},  e_{1, 4} \times e_{3, 5},  e_{1, 4} \times e_{3, 6},  e_{1, 5} \times e_{2, 6},  e_{1, 5} \times e_{3, 6},  e_{1, 5} \times e_{4, 6},  e_{2, 4} \times e_{3, 5},  e_{2, 4} \times e_{3, 6},  e_{2, 5} \times e_{3, 6},  e_{2, 5} \times e_{4, 6},  e_{3, 5}  \times e_{4, 6}  \}$.
\end{enumerate}  
\end{thm}

\begin{proof}  
Since ${\rm Aut}(K_n)$ contains all possible permutations of the vertices, it does not make our job easier to first restrict our search for homomorphisms to those that are induced by automorphisms of the underlying abstract graph.  Thus we use the tools of Subsection~\ref{subsec:tools} rather than those of Subsection~\ref{subsec:LEX}. Drawings of the realizations of $K_5$ and $K_6$, as well justifications of the poset relations and non-relations, appear in Appendix \ref{app:cliques}.
\end{proof}

\begin{figure}[htbp]
 \centerline{\epsfig{file =  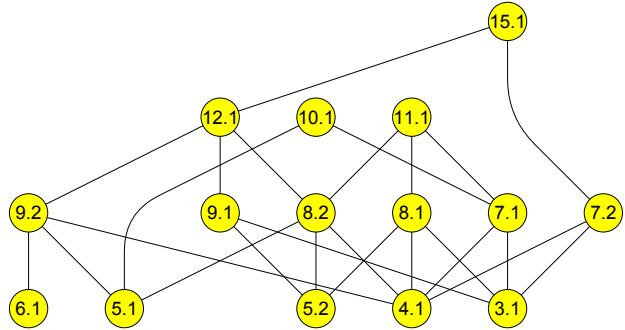,   width = .6\textwidth}}
   \caption{The Hasse diagram for ${\mathcal K}_6$}
   \label{fig:HasseK6poset}
\end{figure}

Observe that $\mathcal{K}_6$ has five minimal elements and three maximal ones. As with cycles, not every possible number of crossings, from 3 up to the maximum of 15, is achieved; there are no realizations of $K_6$ containing 13 crossings or 14 crossings. Clearly, the number of edge crossings cannot act as a rank function. In fact, $\mathcal{K}_6$ is not a graded poset because it has maximal chains between 3.1 and 15.1 of different lengths: 
$3.1 \prec 7.2 \prec 15.1$ and $3.1 \prec 9.1 \prec 12.1 \prec 15.1$. 
Moreover, $\mathcal{K}_6$ is not a lattice, because realizations 3.1 and 4.1 do not have  a unique supremum.

Although $\mathcal{K}_6$ has no rank function, the function taking a realization to the number of vertices in the boundary of  its convex hull is order-preserving.  In Figure~\ref{fig:HasseK6poset}, all realizations displayed on the bottom level of the Hasse diagram have 3 vertices in the boundary of the convex hull, those on the second level have 4, those on the third level have 5 and realization 15.1 has 6.

\begin{thm}\label{prop:chains} \rm
For all $n \geq 3$,  $\mathcal{K}_n$ contains a maximal chain of length $n-2$.
More precisely,  $\mathcal{K}_n$ contains a  chain of the form
$$\overline H_3 \prec \overline H_4 \prec \cdots \prec \overline H_n$$
where $\overline{H}_k$ denotes a geometric realization of $K_n$ with $k$ vertices on the boundary of its convex hull.
\end{thm}

\begin{proof}
We start with a template; consider the circle $x^2 + (y+1)^2 = 4$ in the $xy$ plane, together with the two tangent lines at $(- \sqrt{3}, 0)$ and $(\sqrt{3}, 0)$ that intersect at $(0, 3)$.  Place $n-1$ vertices along the upper portion of the circle, starting at $(- \sqrt{3}, 0)$ and ending at $( \sqrt{3}, 0)$; they should be roughly evenly spaced, but in general position. 
Label the leftmost one $n$, and the remaining ones $2, 3, \dots, n-1$ from left to right.  Add another vertex at $(0, 3)$ and label it $1$. To complete the template, for each $k \in \{2, 3, \dots, n-2\}$, add a ray from vertex $1$ through vertex $k$ and mark where it intersects the lower portion of the circle with $(n-k+1)^*$. See Figure \ref{fig:chains}.

\begin{figure}[htbp]
 \centerline{\epsfig{file =  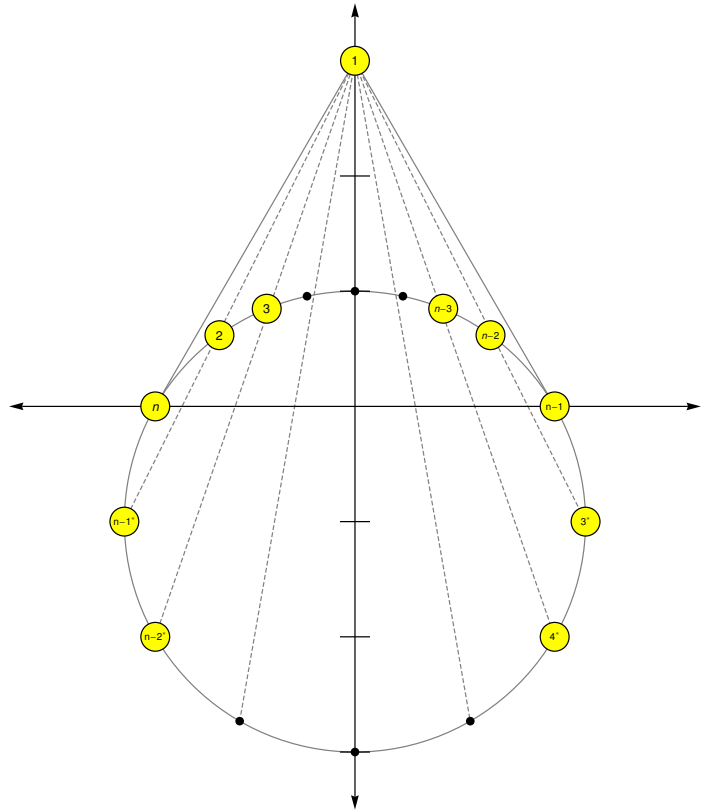,   width = .7\textwidth}}
   \caption{Template for the proof of Theorem \ref{prop:chains}}
   \label{fig:chains}
\end{figure} 

Joining all pairs of vertices in the template with an edge gives us $\overline{H}_3$.  Note that the boundary of its convex hull consists of vertices $1, n$ and $n-1$ and that all crossings in $\overline{H}_3$ occur in the geometric subgraph induced by the vertices $2, 3, \dots, n$.   To get $\overline{H}_4$, slide vertex $n-1$ clockwise along the circle to position $(n-1)^*$.  Then slide vertices $2$ through $n-2$ clockwise `one spot' along the upper portion of the circle.   This is now a geometric realization of $K_n$ in which the boundary of the convex hull consists of the four vertices $1, n-2, n-1$ and $n$.  Since vertices $2, 3, \dots, n$ are still in convex position, no crossings of $\overline{H}_3 $ have been lost.  However the edge $\{1, n-1\}$ now crosses edges $\{2, n\}, \{3, n\}, \dots, \{n-2, n\}$.   Thus 
$\overline{H}_3  \prec \overline{H}_4$.  To get $\overline{H}_5$, move vertex $n-2$ to position $(n-2)^*$ and shift vertices $2$ through $n-3$ clockwise another  `one spot' along the circle.  This gives a geometric realization in which the boundary of the convex hull consists of the vertices  $1, n-3, n-2, n-1$ and $n$.  Again, no edge crossings have been lost, but edge $\{1, n-2\}$ now crosses $\{2, n\}, \{3, n\}, \dots, \{n-3, n\}$. Iterating this process yields the chain described in the theorem.

To prove the maximality of this chain, note that its final element $\overline{H}_n$ has 
all $n$ vertices in convex position and so has
the maximum number of crossings of any realization of $K_n$ and therefore has no successor in  $\mathcal{K}_n$.   
Next, suppose that $f: \overline{H}\to  \overline{H}_3$ is  
an injective geometric homomorphism.
By  Proposition~\ref{prop:obs2},
 all edges incident to vertex $v = f^{-1}(1)$ must be uncrossed.  Let $w, x, y, z$ be any other four vertices in $\overline{H}$; if they are not in convex position, then one of them, say $w$,  must lie in the interior of the convex hull of the other three.  This would imply that the edge $\{v, w\}$ is crossed in $\overline{H}$, a contradiction. Hence all $n-1$ other vertices in $\overline{H}$ lie in convex position, implying that in fact $\overline{H} \cong \overline{H}_3$.
\end{proof}

Each of the posets $\mathcal{K}_4$ and $\mathcal{K}_5$ is precisely the chain given in Theorem \ref{prop:chains}.  Within $\mathcal{K}_6$, the chain constructed in Theorem \ref{prop:chains}  is 
$5.1 \prec 9.2 \prec 12.1 \prec 15.1$.

\section{Open Questions}\label{sect:questions}

\begin{enumerate}
\item Are there (closed or recursive) formulas for the number of elements in $\mathcal{P}_n$ or $\mathcal{C}_n$?
\item  For $3 \leq k \leq n$, is $\mathcal{K}_k$ a sub-poset of $\mathcal{K}_n$?
\item If $\overline{K}_n \prec \widehat{K}_n$,  
must the number of vertices in the convex hull of $\overline{K}_n$ be strictly less than that of $\widehat{K}_n$? If so, then the chain constructed in Theorem \ref{prop:chains} is a maxi\emph{mum} chain.
\item  We saw that $\mathcal{K}_6$ has a maximal chain of length 2.  What is the length of a smallest possible maximal chain in $\mathcal{K}_n$?
\item What is the geometric homomorphism poset for other common families of  graphs?  In \cite{Cockburn2010}, Cockburn has determined the geometric homomorphism poset $\mathcal{K}_{2,n}$ for one family of complete bipartite graphs.
\end{enumerate}

\begin{appendices}
\cleardoublepage
\appendixpage
\section{Geometric Realizations and Posets for the paths $P_5$ and $P_6$}\label{app:paths}

In this appendix we prove that $P_5$ and $P_6$ have the geometric realizations claimed in Theorems~\ref{thm:p1top5posets} and \ref{thm:p6geoms},  and we provide evidence to easily check the correctness of the Hasse diagrams for ${\mathcal P}_5$ and ${\mathcal P}_6$ given in Figures~\ref{fig:P5_poset} and \ref{fig:P6_poset},  respectively.

Table~\ref{tab:p5} gives the four  non-equivalent,  non-plane,  geometric realizations of $P_5$,  using the labels given in Theorem~\ref{thm:p1top5posets};  it follows from Lemma~\ref{lem:p5crossingrule}  that there are no other realizations. It is easy to see from Table~\ref{tab:p5} that the identity map is a geometric homomorphism from each realization to each other that has more crossings,  justifying the Hasse diagram shown in Figure~\ref{fig:P5_poset}.


\begin{table}
\begin{center}
\begin{tabular}{c}
\centerline{\includegraphics[width=\textwidth,keepaspectratio]{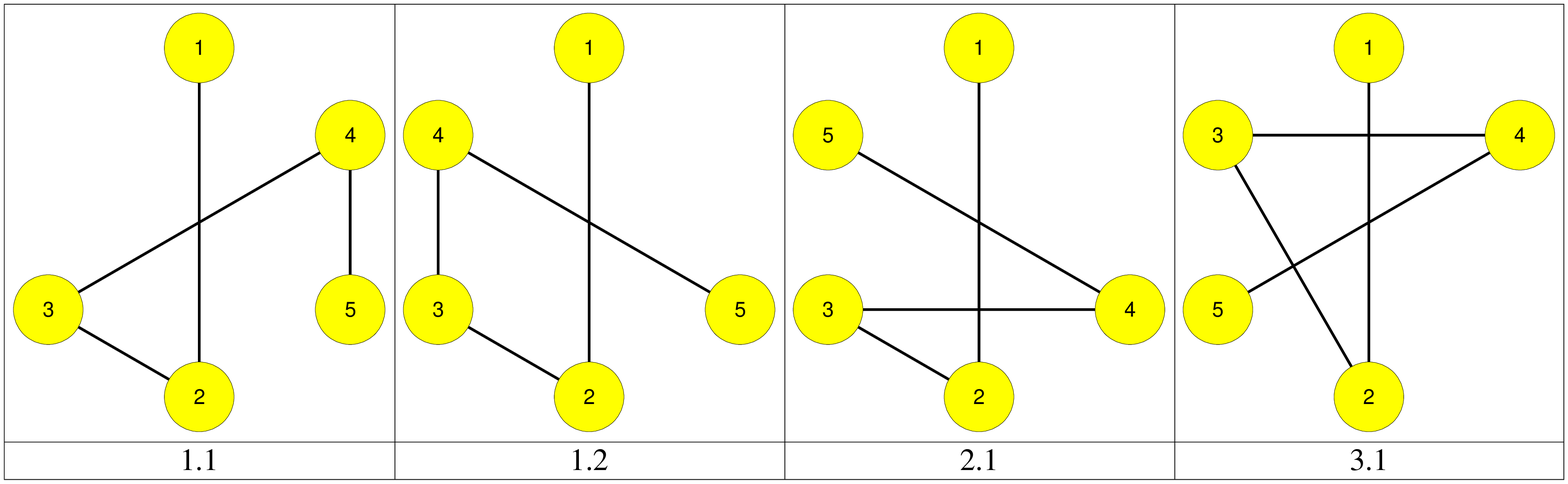}}
\end{tabular}
  \caption{The four non-plane geometric realizations of $P_5$}
  \label{tab:p5}
\end{center}
\end{table}


Table~\ref{tab:p6}  shows the thirty  non-plane geometric realizations of $P_6$,  using the labels given in Theorem~\ref{thm:p6geoms};  it follows from Lemmas~\ref{lem:p5crossingrule}  and \ref{lem:p6rule} that there are no others. 

Table~\ref{tab:p6x} shows the line/crossing graphs corresponding to  each realization in Table~\ref{tab:p6}.
The vertices,  $e_i = \{i,  i+1\},  i = 1,  \ldots,  5$,  run counterclockwise from the top left,  and each vertex is labeled with the number of times the edge in the corresponding geometric realization of $P_6$ is crossed. The red,  dashed edges belong to the line graph,  and the black,  solid edges belong to the crossing graph (so the vertex labels are the degrees in the crossing graph). 

As indicated in Proposition~\ref{prop:preceqiff},  given two geometric realizations  of a graph $G$,  $\G$ and $\hat G$,  the line/crossing graphs are useful tools to determine if $\G$ and $\hat G$ are equivalent,  or if $\G \prec \hat G$. These tools are particularly useful when $G$ is a path $P_n$,  because there are only two isomorphisms to check: the identity map $I_n$ and the map $R_n$ that reverses the order of the vertices. For $P_6$,  we need only observe whether one of these two maps induces a color-preserving injection from 
the line/crossing graph of  one realization $\P_6$ into the line/crossing graph of another realization $\widehat{P}_6$ to determine whether $\P_6 \preceq \widehat{P}_6$. For example,  in Table~\ref{tab:p6x},  compare the line/crossing graph labeled 2.3 with the ones labeled 3.2,  3.3,  and 3.4. It is easy to see that neither the identity nor the reversal map induces a color-preserving injection of 2.3 into 3.2,  but the identity map is a color-preserving injection of 2.3 into 3.3,  and the reversal map is a color-preserving injection of 2.3 into 3.4. Thus $2.3 \not\prec 3.2$,  but $2.3 \prec 3.3$ and $2.3 \prec 3.4$. In this way Table~\ref{tab:p6x} can be used to verify the correctness of  the Hasse diagram for ${\mathcal P}_6$ shown in Figure~\ref{fig:P6_poset}.


\begin{table}
\begin{center}
\begin{tabular}{c}
\centerline{\includegraphics[width=.95\textwidth,keepaspectratio]{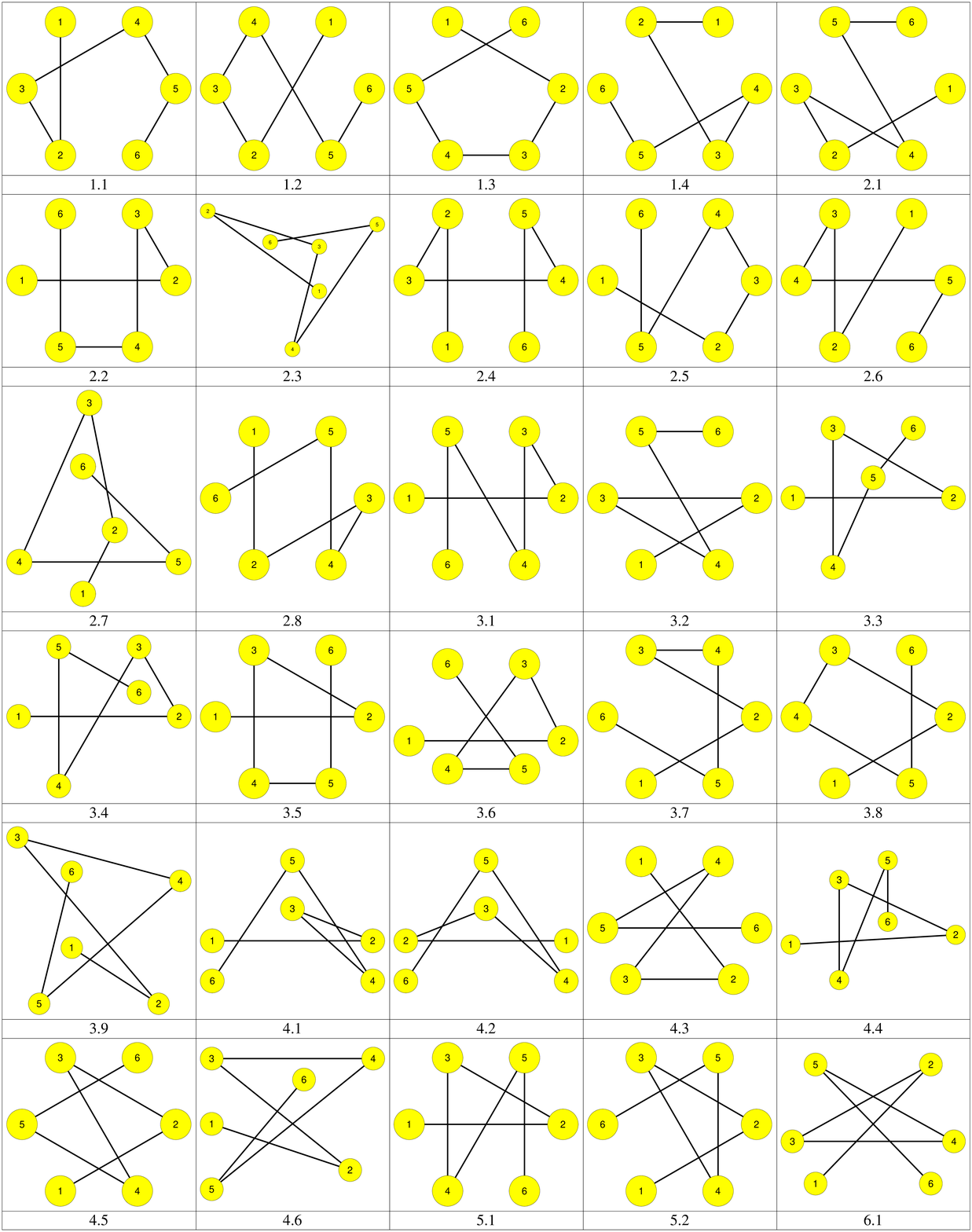}}\end{tabular}
  \caption{The thirty non-plane geometric realizations of $P_6$}
  \label{tab:p6}
\end{center}
\end{table}

\begin{table}
\begin{center}
\begin{tabular}{c}
\centerline{\includegraphics[width=0.98\textwidth,keepaspectratio]{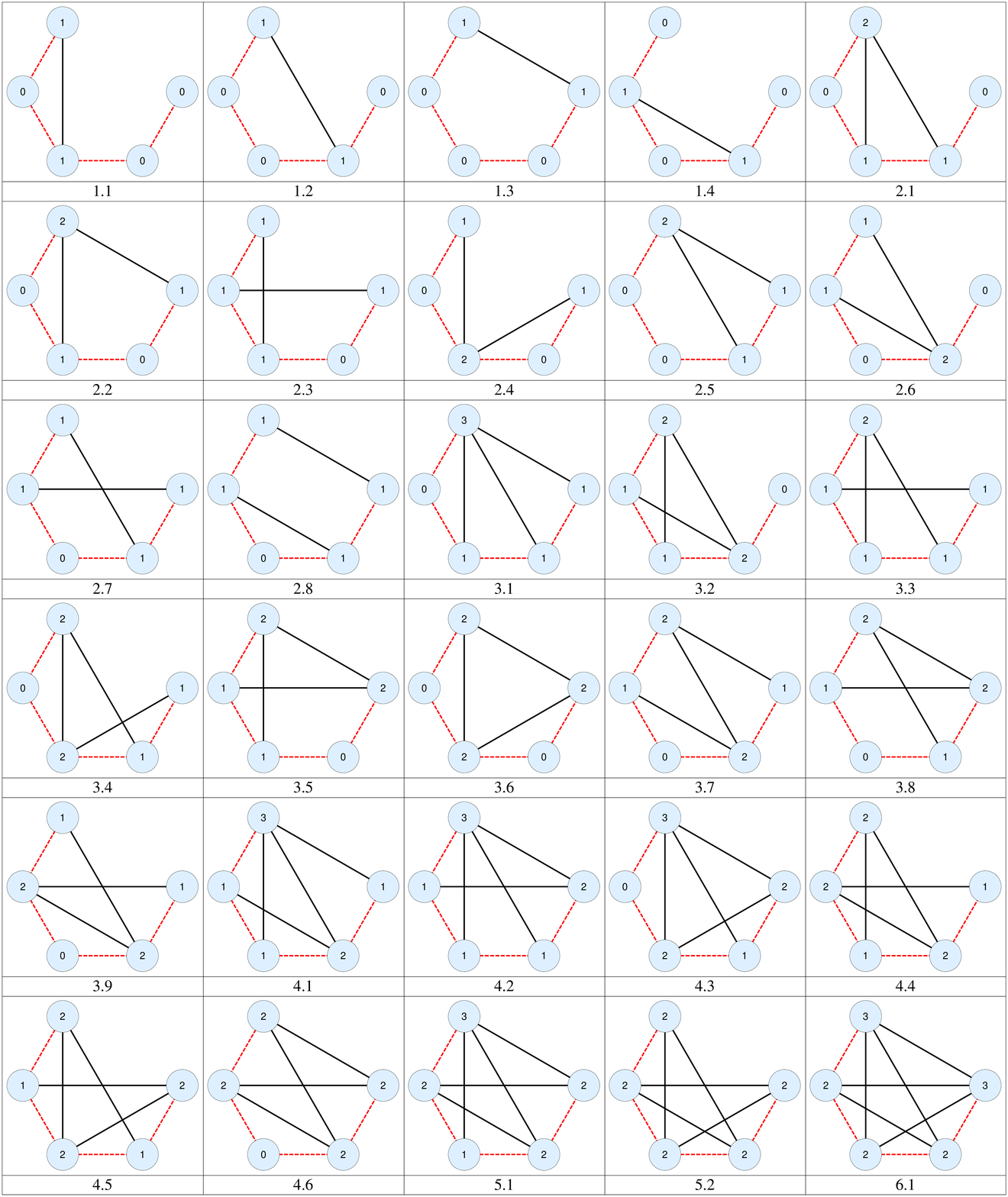}}
\end{tabular}
     \caption{The line/crossing graphs of the geometric realizations  in Table~\ref{tab:p6}}
  \label{tab:p6x}

\end{center}
\end{table}

\section{Geometric Realizations and Posets for the cycles $C_5$ and $C_6$}\label{app:cycles}

In this appendix we prove that $C_5$ and $C_6$ have the geometric realizations claimed in Theorem~\ref{thm:c1toc6posets},  and we provide evidence to easily check the correctness of the Hasse diagram for ${\mathcal C}_6$ given in Figure~\ref{fig:C6_poset}.

Table~\ref{tab:c5}  lists the four  non-equivalent,  non-plane,  geometric realizations of $C_5$,  using the labels given in Theorem~\ref{thm:c1toc6posets};  it follows from Lemma~\ref{lem:p5crossingrule} that there are no other realizations. It is easy to see from Table~\ref{tab:c5} that the identity map is a geometric homomorphism from each realization to each other that has more crossings,  making the Hasse diagram ${\mathcal C}_5$ a chain of four elements.



\begin{table}
\begin{center}
\begin{tabular}{c}
\centerline{\includegraphics[width=0.75\textwidth,keepaspectratio]{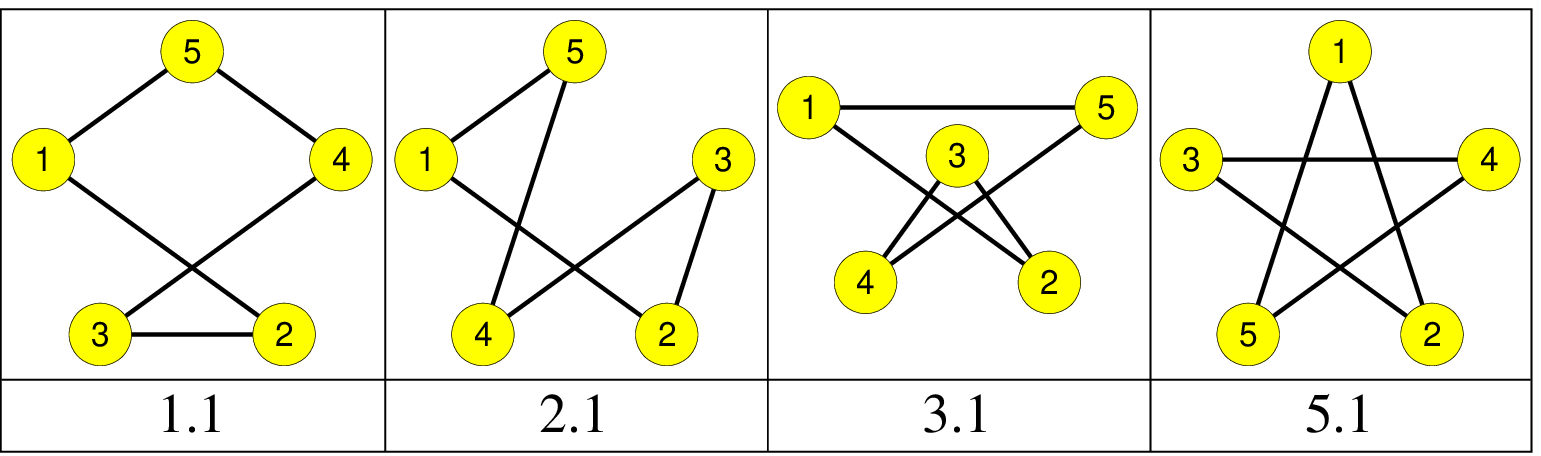}}
\end{tabular}
  \caption{The four non-plane geometric realizations of $C_5$}
  \label{tab:c5}
\end{center}
\end{table}


The structure of the poset $\mathcal{C}_6$ is based in part on Lemma~\ref{lem:c6crossingrules}.  The proof of claim~\ref{cl1} appears in Section~\ref{sect:cycles}; for completeness, we include the proofs of the other claims below. 
\medskip

\noindent {\bf Lemma~\ref{lem:c6crossingrules}.}
Let $\overline{C}_6$ be a geometric realization of the cycle $C_6$,  with edges labeled consecutively,  $e_1,  e_2,  \ldots,  e_6$.
\begin{enumerate}
  \item If $\overline{C}_6$ contains the crossings $e_1\times e_3,  e_1\times e_4$,  and $e_1\times e_5$,  then it doesn't contain the crossing $e_2\times e_6$.
  \item If $\overline{C}_6$ contains the crossings $e_1\times e_3,  e_1\times e_4,  e_1\times e_5,  e_2\times e_4$,  and $e_4\times e_6$,  then it also contains the crossing $e_2\times e_5$.
  \item If $\overline{C}_6$ contains the crossings $e_1\times e_3,  e_1\times e_4,  e_2\times e_4$,  and $e_2\times e_5$,  then it also contains at least one of  the crossings $e_1\times e_5,  e_3\times e_5,  e_3\times e_6$.
  \item If $\overline{C}_6$ contains the crossings $e_1\times e_3,  e_1\times e_4,  e_2\times e_5$,  and $e_4\times e_6$,  then it also contains at least one of  the crossings $e_2\times e_4,   e_3\times e_5,  e_3\times e_6$.
  \item If $\overline{C}_6$ contains the crossings $e_1\times e_3,  e_1\times e_4,  e_2\times e_5$,  and $e_3\times e_6$,  then it also contains at least one of  the crossings $e_2\times e_4,  e_2\times e_6,  e_3\times e_5,  e_4\times e_6$.
\end{enumerate}  

\begin{proof}
The proof of claim 1 appears in the paper. We prove claims 2-5 by contradiction, so for claim 2 assume that we do not have the crossing $e_2\times e_5$: Since we have the crossing $e_1\times e_3$, we may assume that the edge $e_2$ is horizontal and 1 and 4 lie above the line through this edge, as shown in Figure~\ref{pf:cl23}. Because we also have the crosssings $e_1\times e_4$ and $e_2\times e_4$, it follows that vertex 5 is in the region labeled S, which is  the part of the cone with vertex 4 and sides passing through 2 and 3  that lies below edge $e_2$.  Since we have the crossing $e_1\times e_5$ but not $e_2\times e_5$, it follows that vertex 6 is in the region labeled N, which is the part of the cone with vertex 5 and sides passing though 1 and 3 that is above edge $e_1$. But then we cannot have the crossing $e_4\times e_6$, a contradiction.

\begin{figure}[htbp]
 \centerline{\epsfig{file =  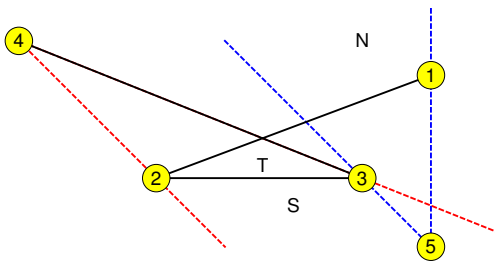,   width = .5\textwidth}}
  \caption{For proof of claims 2 and 3}
  \label{pf:cl23}
\end{figure}

For claim 3, as in the proof of claim 2, having the crossings $e_1\times e_3,  e_1\times e_4,  e_2\times e_4$ implies that vertex 5 lies in the region of Figure~\ref{pf:cl23} labeled S. By asssumption we have neither of the crossings $e_1\times e_5, e_3\times e_5$ but we do have the crossing $e_2\times e_5$, this forces vertex 6 to lie in the region labeled T, which in turn forces the crossing $e_3\times e_6$, a contradiction.

For claim 4, by assumption we have the crossings $e_1\times e_3$ and $e_1\times e_4$ but not  $e_2\times e_4$. It follows that vertex 5 is in one of the regions labeled T and E in Figure~\ref{pf:cl45}, in which the edge $e_2$ is horizontal. First suppose $5\in T$. Since we have crossing $e_2\times e_5$, vertex 6 lies below the horizontal line through edge $e_2$. But then the crossing $e_4\times e_6$ forces the crossing $e_3\times e_6$, a contradiction. So then  $5 \in E$, but now it is impossible to have crossing $e_2\times e_5$ but not $e_3\times e_5$ or $e_3\times e_6$. 

\begin{figure}[htbp]
 \centerline{\epsfig{file =  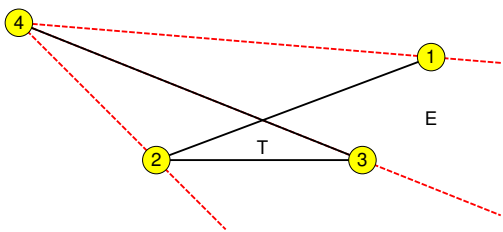,   width = .5\textwidth}}
  \caption{For proof of claims 4 and 5}
  \label{pf:cl45}
\end{figure}

For claim 5, as in the proof for claim 4, we may assume that vertex 5 is in one of the regions in Figure~\ref{pf:cl45} labeled T or E, since by assumption we have crossings $e_1\times e_3$ and $e_1\times e_4$ but not $e_2\times e_4$. If $5\in T$, then crossing $2\times 5$ implies that 6 lies below the horizontal line through edge $e_2$. But then the crossing $e_3\times e_6$ forces either $e_2\times e_6$ or $e_4\times e_6$, a contradiction in either case. So suppose that $5\in E$. In order to have crossing $e_2\times e_5$ but not $e_3\times e_5$,  edge $e_5$ must cross $e_2$ from below, i.e., 5 is below the horizontal line through $e_2$ and 6 is above that line. But the the crossing $e_3\times e_6$ forces the crossing $e_4\times e_6$, again a contradiction. 
\end{proof}

Table~\ref{tab:c6}  shows the twenty-five non-plane geometric realizations of $C_6$,  using the labels given in Theorem~\ref{thm:p6geoms};  it follows from Lemmas~\ref{lem:p5crossingrule},  \ref{lem:p6rule},  \ref{lem:furry},  \ref{lem:insideoutside} and \ref{lem:c6crossingrules} that there are no other realizations.

Table~\ref{tab:c6x} shows the line/crossing graphs corresponding to  each realization in Table~\ref{tab:c6}.
The vertices,  $e_i = \{i,  i+1\},  i = 1,  \ldots,  6$,  run counterclockwise from the top,  and each vertex is labeled with the number of times the edge in the corresponding geometric realization of $C_6$ is crossed. The red,  dashed edges belong to the line graph,  and the black,  solid edges belong to the crossing graph (so the vertex labels are the degrees in the crossing graph). 

As in Appendix~\ref{app:paths},  the line/crossing graphs are again useful tools to determine if $\C_6$ and $\widehat{C}_6$ are equivalent,  or if $\C_6 \prec \widehat{C}_6$. The isomorphisms of $C_6$ are the six rotations,  the reversal map,  and the compositons of the rotations with the reversal map. We need only observe whether one of these maps induces a color-preserving injection of the line/crossing graph of  one realization $\C_6$ to a subgraph of another realization $\widehat{C}_6$ to determine whether $\C_6 \preceq \widehat{C}_6$. For example,  in Table~\ref{tab:c6x},  compare the line/crossing graph labeled 2.2 with the ones labeled 3.1,  3.2,  and 3.4. It is easy to see that the identity map is a color-preserving injection of 2.2 into 3.5,  a rotation is a color-preserving injection from 2.1 into 3.4,  but there is no color-preserving injection from 2.1 into 3.2.

\begin{table}
\begin{center}
\begin{tabular}{c}
\centerline{\includegraphics[width=\textwidth,keepaspectratio]{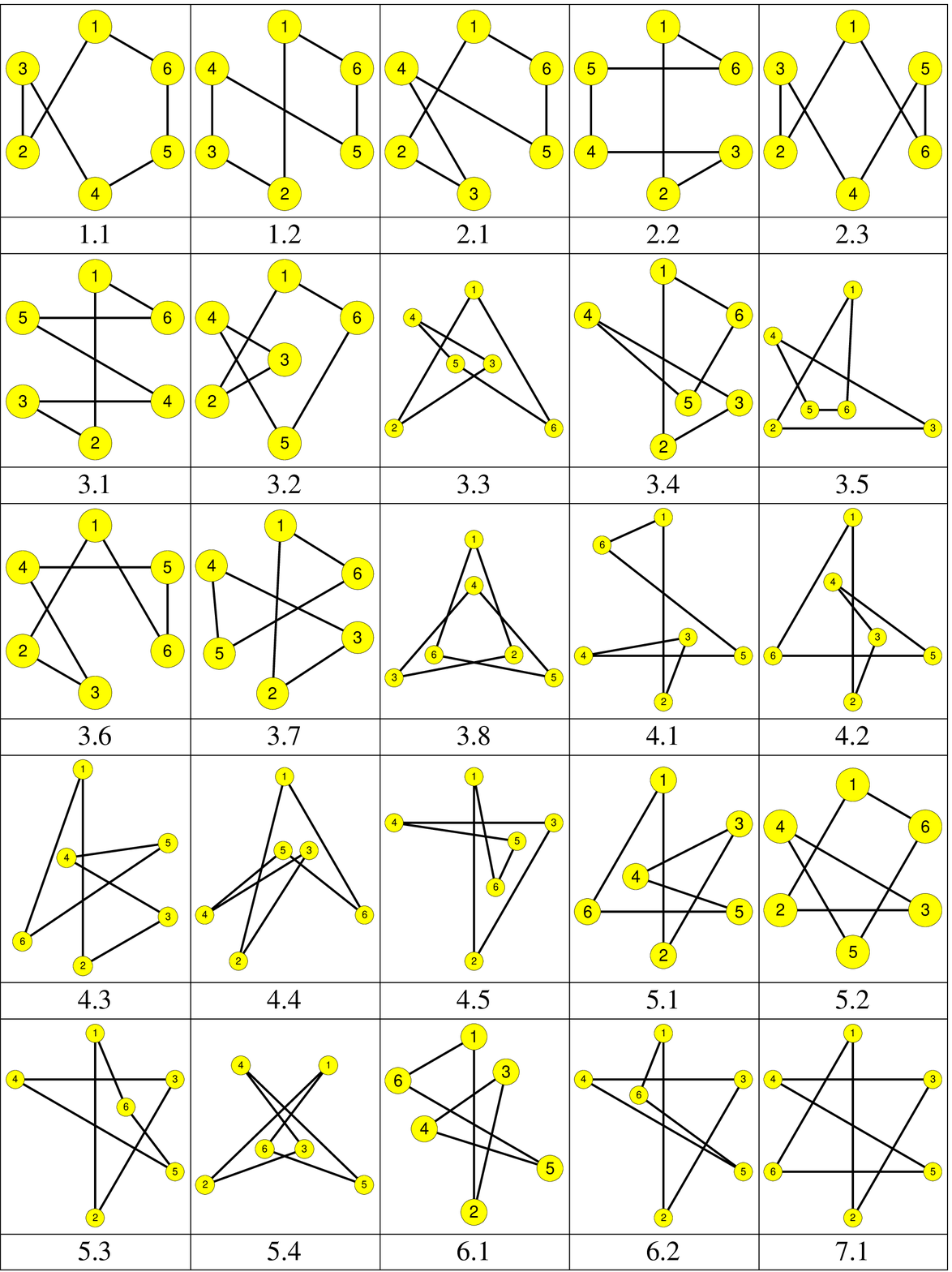}}
\end{tabular}
  \caption{The twenty-five non-plane geometric realizations of $C_6$}
  \label{tab:c6}
\end{center}
\end{table}

\begin{table}
\begin{center}
\begin{tabular}{c}
\centerline{\includegraphics[width=\textwidth,keepaspectratio]{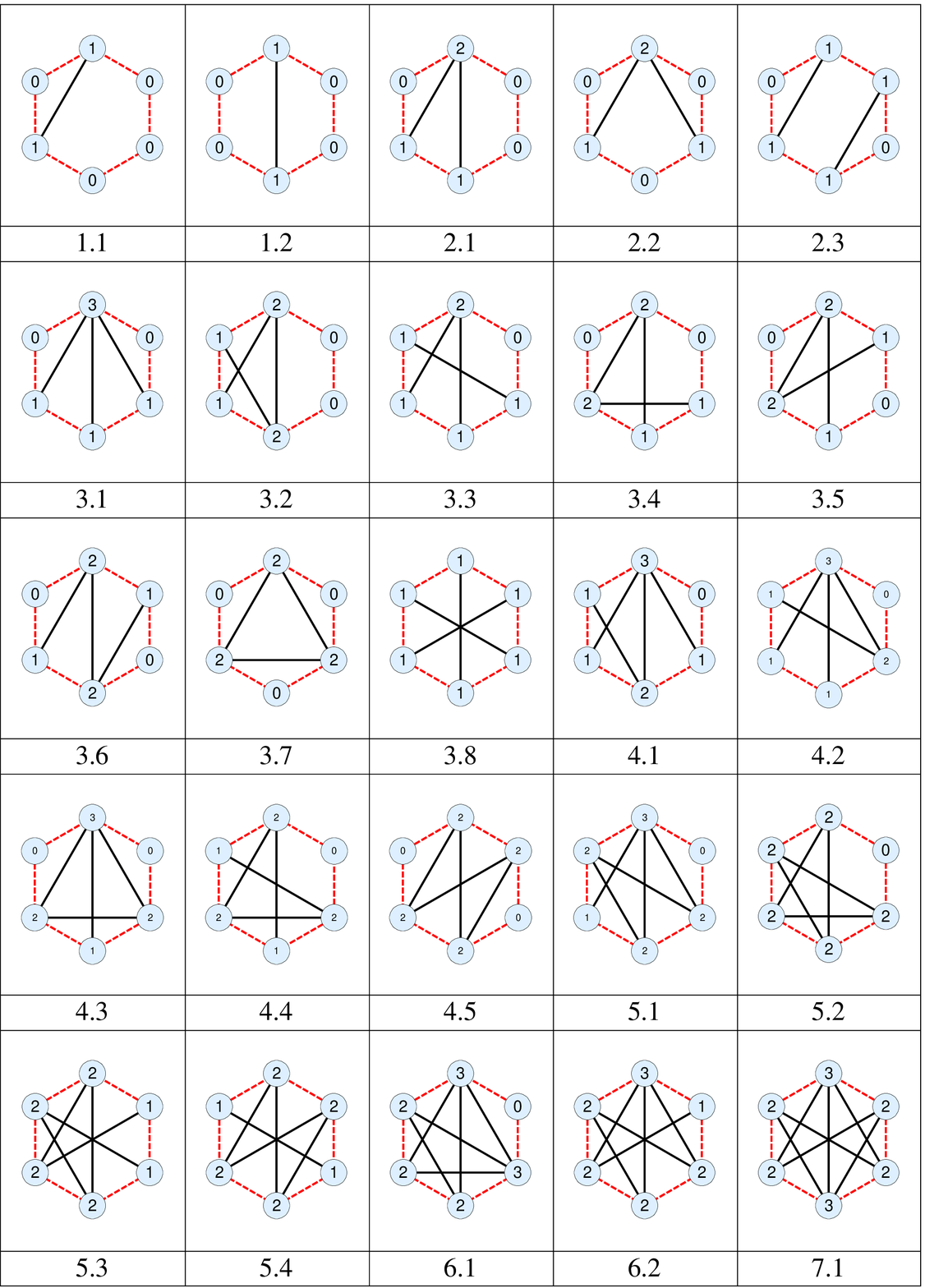}}
\end{tabular}
 \caption{The line/crossing graphs of the geometric realizations in Table~\ref{tab:c6}}
 \label{tab:c6x}
\end{center}
\end{table}

\section{Geometric Realizations and Poset for the clique $K_6$}\label{app:cliques}

In this appendix, we provide the details of the proof of Theorem~\ref{thm:k1tok6posets}.  As noted there, the posets $\mathcal{K}_4$ and $\mathcal{K}_5$ are chains. 
Any vertex bijection is a geometric homomorphism from the plane realization of $K_4$ to the convex realization; with the vertex labeling on the different realizations of $K_5$ given in Figure~\ref{fig:K5Poset}, the identity is a geometric homomorphism from each realization to one with more crossings.

\begin{figure}[htbp] 
\centerline{\epsfig{file =  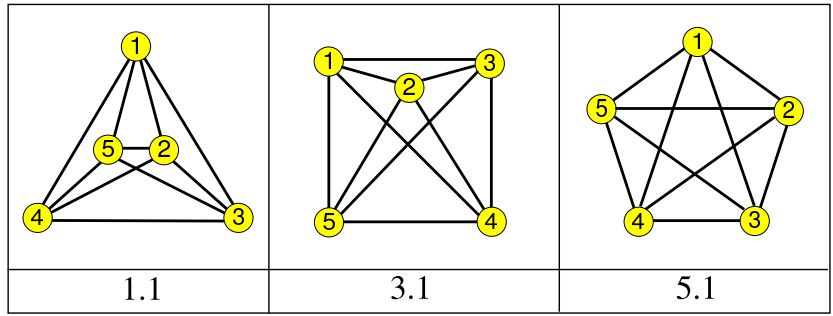,  width = .8\textwidth}}
   \caption{The three realizations of $K_5$.}
   \label{fig:K5Poset}
\end{figure}

Justifying the poset structure of $\mathcal{K}_6$, illustrated by the Hasse diagram in Figure~\ref{fig:HasseK6poset}, requires more work.  Drawings of the realizations are given in Figure~\ref{fig:K6Poset}; with the vertex labelings shown, all covering relations are induced by the identity map, except the eight that induced by the geometric homomorphisms listed in Table~\ref{tab:K6nonids}.

\begin{table}[htdp]
\caption{Non-identity geometric homomorphisms in $\mathcal{K}_6$.}
\begin{center}
\begin{tabular}{|c|c |}
\hline
3.1  & 8.1\\ \hline
1 &  1\\
2 &  6\\
3 &  3\\
4 & 4\\
5 &  5\\
6 & 2 \\ \hline
\end{tabular}
\quad
\begin{tabular}{|c|c c c|}
\hline
4.1 & 7.2 &8.2 & 9.2 \\ \hline
1 & 3 & 6 & 1 \\
2 & 6 & 5 & 5\\
3 & 1 & 2 & 3 \\
4 & 5 & 3 & 4 \\
5 & 4 & 4 & 6 \\
6 & 2 & 1 & 2  \\ \hline
\end{tabular}
\vspace{.1in}

\begin{tabular}{|c|c|}
\hline
5.1 & 10.1 \\ \hline
1 & 1 \\
2 & 2 \\
3 & 3 \\
4 & 4 \\
5 & 6 \\
6 & 5 \\ \hline
\end{tabular}
\quad
\begin{tabular}{|c|c c|}
\hline
5.2 & 8.1 & 8.2\\ \hline
1 & 2 & 6 \\
2 & 3 & 5 \\
3 & 6 & 4 \\
4 & 4 & 3  \\
5 & 5 & 2 \\
6 & 1 & 1 \\ \hline
\end{tabular}
\quad
\begin{tabular}{|c|c c|}
\hline
8.2 & 11.1 & 12.1\\ \hline
1 & 1 & 3 \\
2 & 5 & 2 \\
3 & 4 & 1 \\
4 & 6 & 6 \\
5 & 3 & 5\\
6 & 2 & 4\\ \hline
\end{tabular}
\end{center}
\label{tab:K6nonids}
\end{table}

\begin{figure}[htbp] 
\centerline{\epsfig{file = 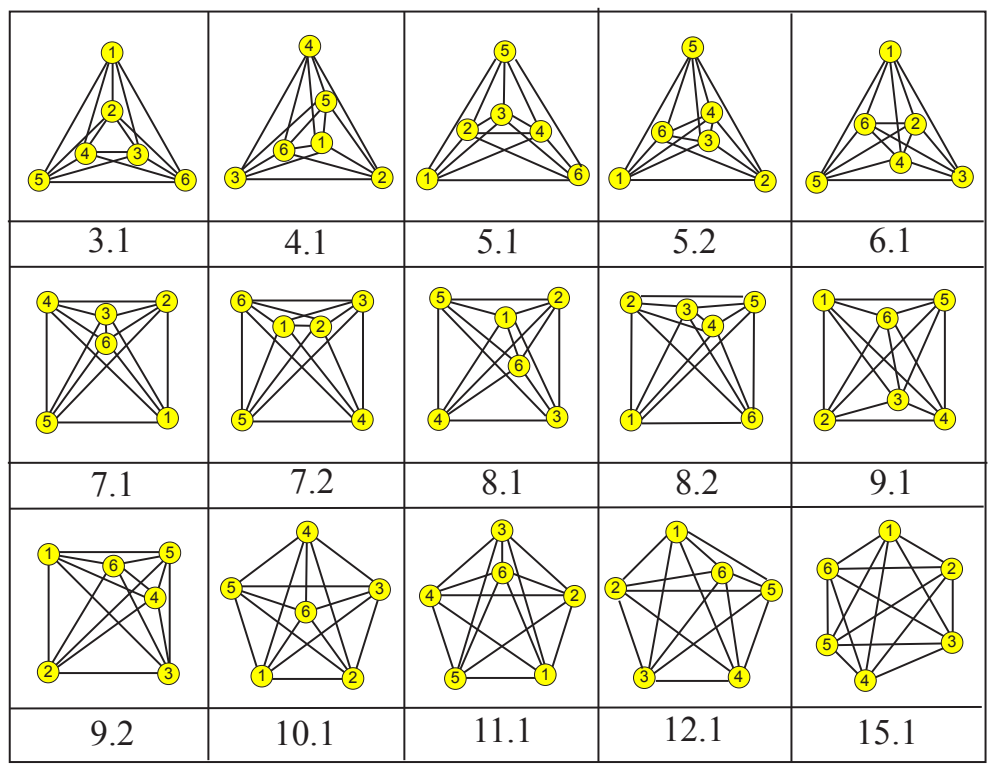,   width =\textwidth}} 
   \caption{The fifteen realizations of $K_6$.}
   \label{fig:K6Poset}
\end{figure}

\newpage

We can justify the absence of covering relations in Figure~\ref{fig:HasseK6poset} by using the contrapositive of various results in Subsection~\ref{subsec:tools}.  To do so, we record the relevant parameters for each element of $\mathcal{K}_6$ in Table~\ref{table:parameters}. Also, to ease the use of part (4) of Proposition~\ref{prop:observe} , we provide the subgraph of uncrossed edges for each of the realizations in Figure~\ref{fig:K60Poset}.

\begin{table}[htdp]
	\caption{Parameters for elements of $\mathcal{K}_6$.}
	\renewcommand{\arraystretch}{1.5}
	\begin{center}
	\begin{tabular}{|c|c|c|c|c|c|c|}
	\hline
	 $\G$ & $cr(\G)$ & $|E_0(\G)|$ & $\widehat{\omega}(\G)$ & $D_0(\G)$ &  $M(\G)$ \\ \hline
	3.1   & 3 & 9 & 4 & [3, 3, 3, 3, 3, 3] &  [1, 1, 1, 1, 1, 1] \\
	4.1  & 4 & 9 &  4 & [4, 3, 3, 3, 3, 2] &  [2, 2, 2, 2, 1, 1] \\
	5.1    & 5 & 10 &  5 & [5, 3, 3, 3, 3, 3] &  [2, 2, 2, 2, 2, 0] \\
	5.2  & 5 & 9 &  4  & [4, 4, 3, 3, 2, 2] &  [3, 3, 2, 2, 2, 2] \\
	6.1   & 6 & 9 & 4  & [4, 4, 4, 2, 2, 2] &  [3, 3, 3, 3, 3, 3] \\ \hline

	7.1    & 7 & 7 & 4 & [3, 3, 3, 2, 2, 1] &  [3, 3, 3, 3, 2, 1] \\
	7.2  & 7 & 7 & 4 & [3, 3, 2, 2, 2, 2] &  [3, 3, 3, 3, 2, 2] \\
	8.1   & 8 & 7 & 4  & [3, 3, 3, 2, 2, 1] &  [4, 4, 3, 3, 2, 2] \\
	8.2  & 8 & 8 &  5 & [4, 3, 3, 2, 2, 2] &  [3, 3, 3, 3, 3, 2] \\
	9.1   & 9 & 8 & 4 & [3, 3, 3, 3, 2, 2] &  [4, 4, 4, 4, 2, 2] \\
	9.2  & 9 & 8 &  5 & [4, 3, 3, 2, 2, 2] &  [4, 4, 3, 3, 3, 3] \\ \hline

	10.1	 & 10& 5 & 5 & [2, 2, 2, 2, 2, 0]  & [3, 3, 3, 3, 3, 1]\\
	11.1 & 11& 6 & 5 & [3, 2, 2, 2, 2, 1] &   [4, 4, 3, 3, 3, 2]\\
	12.1 &12& 7 & 5 & [3, 3, 2, 2, 2, 2] &   [4, 4, 4, 4, 3, 3]\\ \hline

	15.1 &        15& 6 &  6 & [2, 2, 2, 2, 2, 2] &   [4, 4, 4, 4, 4, 4]\\
	\hline
	\end{tabular}
	\end{center}
	\label{table:parameters}
\end{table}

\begin{figure}[htbp] 
\centerline{\epsfig{file = 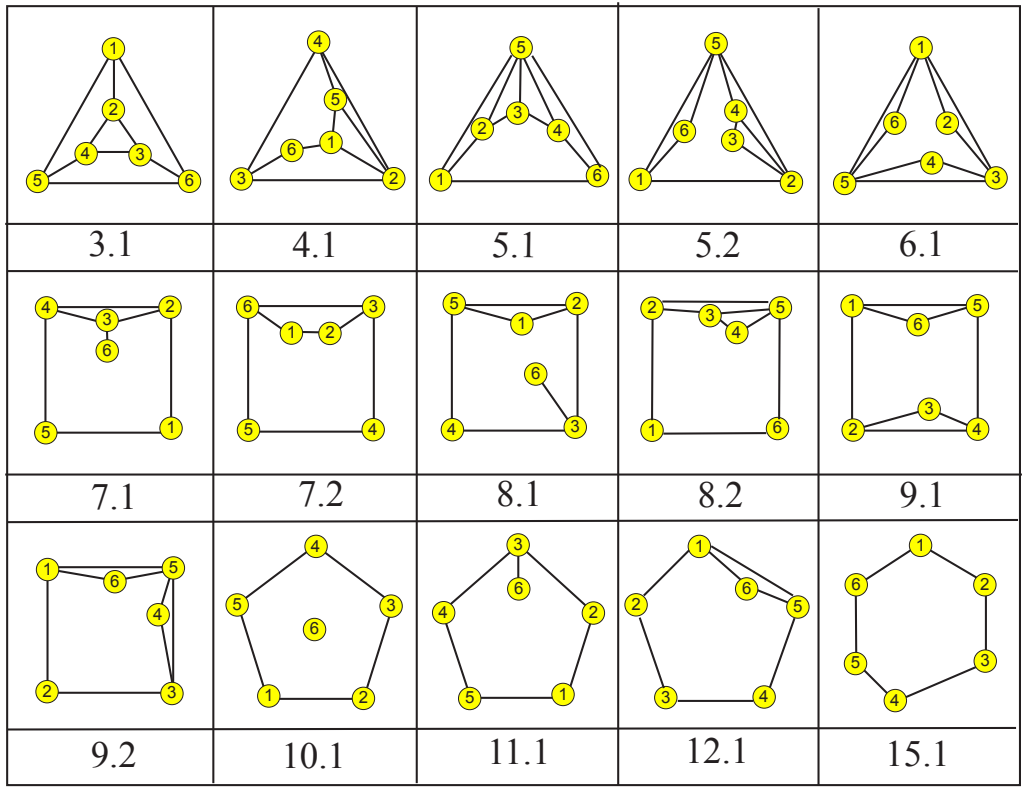,   width =\textwidth}} 
   \caption{The subgraphs of uncrossed edges of realizations of $K_6$.}
   \label{fig:K60Poset}
\end{figure}

\newpage

Justifications for cases of nonprecedence in the poset $\mathcal{K}_6$ are recorded in Table \ref{table:nonprecedence}; each part of Proposition~\ref{prop:observe} and each part of Corollary~\ref{cor:vector} is used at least once.  In some cases, there are several ways of justifying nonprecedence but only one is given.
The following examples indicate how to read this table:
\begin{itemize}
\item the blank in the `$(4.1, 3.1)$' entry means that $4.1 \not \prec 3.1$ is justified simply by the total number of edge crossings (that is, using part (1) of Proposition~\ref{prop:observe}),
\item ``\ref{cor:vector}(\ref{cor1})" in the `$(3.1, 4.1)$' entry means that $3.1 \not \prec 4.1$ is justified by an application of part (1) of Corollary \ref{cor:vector};
\item ``\ref{prop:observe}(\ref{c})" in the `$(3.1, 5.1)$' entry means that $3.1 \not \prec 5.1$ is justified by an application of part (3) of Proposition \ref{prop:observe};
\item ``$\prec$" in the `$(3.1, 8.1)$' entry means that $3.1$ is covered by $8.1$; that is, there is a geometric homomorphism 
 $3.1 \to 8.1$ (this one is recorded in Table \ref{tab:K6nonids}), but no realization $\overline{K}_6$ such that 
 $3.1 \prec \overline{K}_6 \prec 8.1$;
 \item  ``$ \circ $'' in the `$(3.1, 10.1)$' entry means that a geometric homomorphism $3.1 \to 10.1$ can be obtained by composing two geometric homomorphisms (in this case, the two identity maps $3.1 \to 7.1$ and $7.1 \to 10.1$).
 \end{itemize}

\begin{sidewaystable}[htdp]
\renewcommand{\arraystretch}{1.5}
\caption{Justifications for nonprecedence in $\mathcal{K}_6$.}
\begin{center}
\begin{tabular}{|c||c|c|c|c|c||c|c|c|c|c|c||c|c|c||c|}
\hline
        & 3.1 & 4.1 & 5.1& 5.2 & 6.1 & 7.1 & 7.2 & 8.1& 8.2 & 9.1 & 9.2 & 10.1 & 11.1 & 12.1 & 15.1 \\ \hline \hline
 3.1 & =  &  \ref{cor:vector}(\ref{cor1}) & \ref{prop:observe}(\ref{c}) &  \ref{cor:vector}(\ref{cor1}) & \ref{cor:vector}(\ref{cor1}) &  $\prec$ &  $\prec$ & $\prec$  &  \ref{cor:vector}(\ref{cor1})  & $\prec$  &  \ref{cor:vector}(\ref{cor1})  & $\circ$  &   $ \circ $ & $ \circ $ & $ \circ$   \\ \cline{2-16}
 4.1& \multicolumn{1}{c |}{}   & = &  \ref{prop:observe}(\ref{c}) & \ref{cor:vector}(\ref{cor1}) &  \ref{cor:vector}(\ref{cor1}) & $\prec$  & $\prec$  &  $\prec$ & $\prec$  & \ref{prop:observe}(\ref{d})  &  $\prec$ & $\circ$ &   $ \circ $ & $ \circ $ & $ \circ$  \\ \cline{3-16}
 5.1 & \multicolumn{2}{c |}{}    & = & \ref{prop:observe}(\ref{e}) & \ref{prop:observe}(\ref{e}) &  \ref{prop:observe}(\ref{e}) & \ref{prop:observe}(\ref{e})  & \ref{prop:observe}(\ref{e})   &  $\prec$ & \ref{prop:observe}(\ref{e})  &  $\prec$ & $\prec$ & $ \circ $ & $ \circ$  &  $\ \circ$     \\ \cline{4-16}
 5.2 & \multicolumn{2}{c |}{}     & \ref{prop:observe}(\ref{c}) & = & \ref{cor:vector}(\ref{cor1}) & \ref{cor:vector}(\ref{cor2}) & \ref{prop:observe}(\ref{d})+(\ref{cor2})  & $\prec$  & $\prec$  & $\prec$  & \ref{prop:observe}(\ref{d}) & \ref{cor:vector}(\ref{cor2}) & $ \circ $  &$ \circ $  &  $ \circ$   \\ \cline{4-16}
 6.1 & \multicolumn{4}{c |}{}   & = &  \ref{cor:vector}(\ref{cor2}) & \ref{cor:vector}(\ref{cor2}) &  \ref{cor:vector}(\ref{cor2}) &  \ref{cor:vector}(\ref{cor2}) &  \ref{cor:vector}(\ref{cor1}) & $\prec$  &  \ref{cor:vector}(\ref{cor2}) & \ref{cor:vector}(\ref{cor2})  & $ \circ $ & $ \circ$   \\  \hline
 \hline 
 7.1 & \multicolumn{5}{c ||}{}   & = &  \ref{cor:vector}(\ref{cor1}) & \ref{prop:observe}(\ref{d}) & \ref{prop:observe}(\ref{c}) &  \ref{prop:observe}(\ref{c})&  \ref{prop:observe}(\ref{c}) & $\prec$  & $\prec$   &  \ref{cor:vector}(\ref{cor1}) &  \ref{cor:vector}(\ref{cor1})   \\ \cline{7-16}
 7.2 &  \multicolumn{5}{c ||}{}  & \ref{cor:vector}(\ref{cor2}) & = &  \ref{cor:vector}(\ref{cor1}) & \ref{prop:observe}(\ref{c}) &  \ref{prop:observe}(\ref{c}) &  \ref{prop:observe}(\ref{c}) &  \ref{cor:vector}(\ref{cor2}) & \ref{prop:observe}(\ref{d})   &  \ref{prop:observe}(\ref{d})  &   $\prec$  \\ \cline{7-16}
 8.1 &  \multicolumn{7}{c |}{}      & = & \ref{prop:observe}(\ref{c}) & \ref{prop:observe}(\ref{c}) &  \ref{prop:observe}(\ref{c}) &  \ref{cor:vector}(\ref{cor2}) & $\prec$   &  \ref{prop:observe}(\ref{d}) &  \ref{prop:observe}(\ref{d})  \\ \cline{9-16}
 8.2 &  \multicolumn{7}{c |}{}   &  \ref{cor:vector}(\ref{cor2}) & = &  \ref{prop:observe}(\ref{e}) &  \ref{prop:observe}(\ref{d}) &  \ref{cor:vector}(\ref{cor2}) &  $\prec$  &  $\prec$ &  $\circ$    \\ \cline{9-16}
 9.1 &  \multicolumn{9}{c |}{}    & = &  \ref{cor:vector}(\ref{cor2}) &  \ref{cor:vector}(\ref{cor2}) &  \ref{cor:vector}(\ref{cor2})  & $\prec$  & $ \circ $   \\ \cline{11-16}
 9.2 &  \multicolumn{9}{c |}{}    &  \ref{prop:observe}(\ref{e}) & = &  \ref{cor:vector}(\ref{cor2}) &  \ref{cor:vector}(\ref{cor2})  & $\prec$  &  $\circ $   \\ \cline{10-16}
 \hline \hline
 10.1 &  \multicolumn{11}{c ||}{}   & = &  \ref{prop:observe}(\ref{c})  & \ref{prop:observe}(\ref{c}) &  \ref{prop:observe}(\ref{c})   \\ \cline{13-16}
 11.1 & \multicolumn{12}{c |}{}    &  = &  \ref{prop:observe}(\ref{c})   & \ref{prop:observe}(\ref{d})  \\ \cline{14-16}
 12.1 &  \multicolumn{13}{c |}{}     & = &  $\prec$   \\ \cline{15-16}
 15.1 & \multicolumn{14}{c ||}{}  & = \\ 
 \hline
\end{tabular}
\end{center}
\label{table:nonprecedence}
\end{sidewaystable}

One entry in this table deserve further elaboration.  To  show $5.2 \not \prec 7.2$, we first look at the uncrossed subgraphs in Figure \ref{fig:K60Poset}.  Note that the uncrossed subgraph of $7.2$ consists of a 6-cycle with one `diameter' chord (the edge $e_{3,6}$); the uncrossed subgraph of  $5.2$ has only one 6-cycle, namely $(1,2,3,4,5,6)$, which has only one `diameter' chord, namely $e_{2,5}$.  
Hence, the only possible pre-image of the uncrossed four-cycle  $(1, 2, 3, 6)$ of $7.2$ must be either the uncrossed  4-cycle $(2, 3, 4, 5)$ or the uncrossed 4-cycle $(1, 2, 5, 6)$ of $5.2$.
Now, in $7.2$,  $e_{1,3}$ and $e_{2,6}$ cross only each other, so  $cr(e_{1,3}) = cr(e_{2,6}) = 1$.  However, in $5.2$, there are only two possible pre-images of this pair of edges, namely $e_{2,4}, e_{3,5}$ or $e_{1,5}, e_{2,6}$.  Since $cr (e_{3,5}) =cr (e_{2,6}) = 2$, by part (2) of Proposition \ref{prop:observe}, $5.2 \not \prec 7.2$.

\end{appendices}
\cleardoublepage

\bibliographystyle{plain}
\bibliography{GeomPosetPaper}

\end{document}